\documentclass[a4paper,11pt,reqno]{article}
\usepackage{amsmath}
\usepackage{relsize}
\usepackage[noadjust]{cite}
\usepackage{color}
\usepackage[dvipsnames]{xcolor}
\usepackage{tikz}
\usetikzlibrary{positioning,chains,fit,shapes,calc}
\usepackage{caption}
\usepackage{hyperref, enumitem}
\hypersetup{
  colorlinks   = true, 
  urlcolor     = blue, 
  linkcolor    = Purple, 
  citecolor   = red 
}
\usepackage{amsmath,amsthm,amssymb}

\usepackage[margin=2.6cm]{geometry}
\usepackage{caption}
\usepackage{subcaption}

\usepackage{stmaryrd}

\usepackage{array}
\newcolumntype{x}[1]{>{\centering\arraybackslash\hspace{0pt}}p{#1}}

\theoremstyle{definition}
\newtheorem{theorem}{Theorem}[section]
\newtheorem{definition}[theorem]{{{Definition}}}
\newtheorem{example}[theorem]{{{Example}}}

\newtheorem{remark}[theorem]{{{Remark}}}

\newtheorem{corollary}[theorem]{{{Corollary}}}
\newtheorem{proposition}[theorem]{{{Proposition}}}
\newtheorem{lemma}[theorem]{{{Lemma}}}

\newcommand{\numberset}{\mathbb}
\newcommand{\N}{\numberset{N}}

\newcommand{\Z}{\numberset{Z}}

\newcommand{\F}{\numberset{F}}

\newcommand{\C}{\mathcal{C}}
\newcommand{\mC}{\mathcal{C}}

\newcommand{\mP}{\mathcal{P}}

\newcommand{\D}{\mathcal{D}}

\newcommand{\mO}{{\mathcal{O}}}
\newcommand{\I}{\mathcal{I}}
\newcommand{\kL}{\mathcal{L}}

\newcommand{\mZ}{\mathcal{Z}}

\newcommand{\Fq}{\F_q}
\newcommand{\Fm}{\F_{q^m}}

\newcommand{\mL}{\mathcal{L}}
\newcommand{\mF}{\mathcal{F}}
\newcommand{\mU}{\mathcal{U}}

\newcommand{\rk}{\textnormal{rk}}

\newcommand{\Hyp}{\mathrm{Hyp}}

\DeclareMathOperator{\GL}{GL}

\DeclareMathOperator{\cl}{cl}
\DeclareMathOperator{\cyc}{cyc}
\DeclareMathOperator{\Cl}{Cl}
\DeclareMathOperator{\Cyc}{Cyc}

\title{The Cyclic Flats of a $q$-Matroid}
\usepackage{authblk}
\author[1,2]{Gianira N. Alfarano}
\affil[1]{Department of Mathematics and Computer Science, Eindhoven University of Technology, the Netherlands, g.n.alfarano@tue.nl}
\affil[2]{Institute of Mathematics, University of Zurich, Switzerland, gianiranicoletta.alfarano@math.uzh.ch}

\author[3]{Eimear Byrne}
\affil[3]{School of Mathematics and Statistics, University College Dublin, Belfield, Ireland, ebyrne@ucd.ie}
\date{}

\begin{document}

\maketitle
\begin{abstract}
   In this paper we develop the theory of cyclic flats of $q$-matroids. We show that the cyclic flats, together with their ranks, uniquely determine a $q$-matroid and hence derive a new $q$-cryptomorphism.
   We introduce the notion of $\F_{q^m}$-independence of an $\F_q$-subspace of $\F_q^n$ and we show that $q$-matroids generalize this concept, in the same way that matroids generalize the notion of linear independence of vectors over a given field.
\end{abstract}

\section*{Introduction}
The concept of $q$-matroid may be traced back to Crapo's PhD thesis \cite{crapo1964theory}. More recently, the relation between rank-metric codes and $q$-matroids has led to these combinatorial objects getting a lot of attention from researchers; see for instance \cite{jurrius2018defining, gorla2019rank, shiromoto2019codes, ghorpade2020polymatroid, byrne2022constructions, byrne2020cryptomorphisms, byrne2021weighted, ceria2021direct}. Indeed, it is well-known that $q$-matroids generalize $\F_{q^m}$-linear rank-metric codes, just as classical matroids generalize linear codes in the Hamming metric.

As in traditional matroid theory, there are many equivalent ways to describe a $q$-matroid axiomatically, which are called $q$-\emph{cryptomorphisms}. A full exposition of these is given in \cite{byrne2020cryptomorphisms}, in terms of rank function, independent spaces, flats, circuits, bases, spanning spaces, the closure function, hyperplanes, open spaces etc.

In matroid theory, one of the most crucial objects is the \emph{lattice of flats} $\mF(M)$ of a matroid~$M$, since it uniquely determines the matroid. Another lattice with the same property is the lattice of \emph{cyclic sets}. This led many researchers to investigate the intersection between these lattices, namely the collection of \emph{cyclic flats} of the matroid.
Cyclic flats have also played several important roles such as in the work of Brylawski, who showed in \cite{brylawski1975affine} that the cyclic flats of a matroid, together with their ranks uniquely determine the matroid. Moreover, Eberhardt showed that they provide the Tutte polynomial in \cite{eberhardt2014computing}, and Bonin and de Mier showed in \cite{bonin2008lattice} that every lattice is isomorphic to the lattice of cyclic flats of a matroid. Applications to coding theory have been recently investigated: it has been proved that many central invariants in coding theory can be naturally described in terms of the lattice of cyclic flats of the associated matroid; see \cite{freij2021cyclic, westerback2016combinatorics}. 
Furthermore, the lattice of cyclic flats of classical polymatroids has also been studied; see \cite{westerback2015applications, westerback2018polymatroidal, csirmaz2020cyclic}.

In this paper we consider a $q$-analogue of this theory: we define the lattice of cyclic flats of a $q$-matroid and show that it, along with the ranks of its elements, determines the $q$-matroid. We furthermore propose a new cryptomorphism of $q$-matroids based on cyclic flat axioms. We consider the codes associated with representable $q$-matroids and show that the cyclic spaces of a $q$-matroid are supports of elements of the corresponding dual code.

The outline of this paper is as follows. In Section \ref{sec:cyclicflats},
a cyclic flat is defined along with a \emph{cyclic operator} and, as for a flat, with a \emph{closure operator}.
We use the two operators to show that the collection of cyclic flats of a $q$-matroid is a lattice that is not induced by its lattice of flats, nor its cyclic spaces, nor the lattice of subspaces of the ground space. We show that the lattice of cyclic flats, together with their ranks, fully determines the $q$-matroid. In Section \ref{sec:cryptomorphism}, we exploit the theory of cyclic flats to establish a new $q$-cryptomorphism and we provide a necessary and sufficient condition for a lattice $\mZ$ of subspaces endowed with a function $r_\mZ:\mZ\to\Z$ to be the lattice of cyclic flats of a $q$-matroid.
In Section \ref{sec:rank-metric}, we further investigate the connection between rank-metric codes and $q$-matroids. We introduce the notion of $\F_{q^m}$-\emph{independence} of an $\F_q$-subspace of $\F_q^n$ and we show that $q$-matroids generalize this concept, in the same way that matroids generalize the notion of linear independence of vectors over a given field. We conclude with a short discussion in Section \ref{subsec:q-polymatroids} on the $q$-polymatroid case.

\paragraph{Notation} Throughout this paper, $n$ denotes a fixed positive integer, $q$ is a prime power, and $\F_q$ denotes the finite field of order $q$. We denote by $E$ a fixed $n$-dimensional vector space over $\F_q$ and by $\mathcal{L}(E)$ the lattice of subspaces of $E$, ordered with respect to inclusion. We write $A\leq B$ to indicate that $A$ is a subspace of $B$. If a subspace is to be understood as being $1$-dimensional, we represent it by a lowercase letter, so for instance, $x\leq A$ means that $x$ is a one-dimensional subspace of $A$. For every $A \in \mL(E)$, we write
$\mP(A)$ to denote its collection of $1$-dimensional subspaces. For every $A\in\mL(E)$, we denote by $\Hyp(A)$ the set of hyperplanes of $A$, i.e. the set of codimension-$1$ subspaces of $A$. These are not to be confused with hyperplanes in the sense of matroids/$q$-matroids. The standard basis  of the space $\F_q^n$ is denoted by $\{e_1,\ldots, e_n\}$. Finally, for a space $A\leq E$, we denote by $A^\perp$ the orthogonal complement of $A$ in $E$, with respect to a fixed non-degenerate symmetric bilinear form.

\section{Background}\label{sec:background}

In this section, we recall some preliminary notions on $q$-matroids and rank-metric codes.
The following definition of $q$-matroid is given in terms of a rank function; see \cite{jurrius2018defining}. Notice that this definition does not require $E$ to be
a vector space over a finite field, however, we will assume that a $q$-matroid is an object defined with respect to an $\F_q$-vector space.

 \begin{definition}\label{rankfunction}
	A $q$-matroid $M$ is a pair $(E,r)$ where
	$r$ is an integer-valued
	function defined on $\mL(E)$ with the following properties:
	\begin{itemize}
		\item[(R1)] Boundedness: $0\leq r(A) \leq \dim A$, for all $A\in \kL(E)$.
		\item[(R2)] Monotonicity: $A\leq B \Rightarrow r(A)\leq r(B)$, for all $A,B\in\mL(E)$.
		\item[(R3)] Submodularity: $r(A+ B)+r(A\cap B)\leq r(A)+r(B)$, for all $A,B\in\mL(E)$.
	\end{itemize}
	The function $r$ is called the {\bf rank function} of the $q$-matroid.
\end{definition}

Given a $q$-matroid $(E,r)$, we define the \textbf{nullity} function $\nu$ to be
$$ \nu: \mL(E)\to \Z, \ A \mapsto \dim(A)-r(A).$$
From Definition \ref{rankfunction}, it follows that the nullity has the following properties:
\begin{itemize}
    \item[(n1)] $\nu(A)\geq 0$, for all $A\in\mL(E)$.
    \item[(n2)] $A\leq B \Rightarrow \nu(A)\leq \nu(B)$, for all $A,B\in\mL(E)$.
    \item[(n3)] $\nu(A+B)+\nu(A\cap B) \geq \nu(A)+\nu(B)$, for all $A,B\in\mL(E)$.
\end{itemize}

 	\begin{definition}\label{independentspaces}
		Let $(E,r)$ be a $q$-matroid.
		A subspace $A\in\mL(E)$ is called an {\bf independent} space of $(E,r)$ if $r(A)=\dim A$.
		We write $\I_r$ to denote the set of independent spaces of the $q$-matroid $(E,r)$:
		\[\I_r :=\{ I \in \kL(E) \mid \dim(I) = r(I) \}. \]
		 If $x\leq E$ and $r(x)=0$, then $x$ is called \textbf{loop} of $M$.
		A subspace that is not an independent space of $(E,r)$ is called a
		{\bf dependent space} of $(E,r)$.
		We call $C \in \kL(E)$ a {\bf circuit} if it is itself a dependent space and every proper subspace of $C$ is independent. We write $\D_r$ and $\C_r$ to denote the sets of dependent spaces and the circuits
		of the $q$-matroid $(E,r)$, respectively.
		A subspace is called an {\bf open space} of $(E,r)$ if it is a (vector space) sum of circuits. We write ${\mO}_r$ to denote the set of open spaces of $(E,r)$.
	\end{definition}
	
	We define a closure operator as follows (c.f. \cite[Definition 5]{byrne2022constructions}).

    \begin{definition}\label{def-closure}
        Let $(E,r)$ be a $q$-matroid.
        For each $A \in \kL(E)$, define $$\Cl_r(A):=\{x \in \kL(E)\mid r(A+x)=r(A)\}.$$
    	The {\bf closure function} of a $q$-matroid $(E,r)$ is the function
    	defined by
    	\[\cl_r:\mathcal{L}(E) \to\mathcal{L}(E), \  A \mapsto \cl_r(A)=\sum_{x \in \Cl_r(A)} x .
    	\]
    \end{definition}

    \begin{definition}\label{def:flat}
	A subspace $A$ of a $q$-matroid $(E,r)$ is called a {\bf flat} or {\bf closed space} if for all 
	$x \in \kL(E)$ such that $x\nleq A$, we have $$r(A+x)>r(A).$$
	We write ${\mathcal F}_r$ to denote the set of flats of the $q$-matroid $(E,r)$, that is
	\[ \mF_r:= \{ A \in \kL(E) \mid r(A+x)>r(A) \:\:\forall x \in \kL( E), \ x \nleq A \}. \]

\end{definition}
\noindent
If it is clear from the context, we will simply write $\I,\D,\C,\mO, \cl $ in place of
$\I_r,\D_r,\C_r, \mO_r,\cl_r$.

In \cite{byrne2020cryptomorphisms}, several cryptomorphisms of $q$-matroids have been established, which give equivalent ways of defining a $q$-matroid.
Here, we recall the \emph{closure function axioms}. The closure function of a $q$-matroid $(E,r)$ necessarily satisfies the closure function axioms; \textcolor{black}{see \cite[Theorem 68]{jurrius2018defining}.}

\begin{definition}\label{closure-axioms}
Let $\cl:\mathcal{L}(E)\to \mathcal{L}(E)$ be a map. We define the following {\bf closure axioms}.
\begin{itemize}
\item[(Cl1)] $A\leq \cl(A)$, for all $A\in\mL(E)$.
\item[(Cl2)] $A\leq B\Rightarrow \cl(A)\leq \cl(B)$, for all $A,B\in\mL(E)$.
\item[(Cl3)] $\cl(A)=\cl(\cl(A))$, for all $A\in\mL(E)$.
\item[(Cl4)] For all $x,y,A \in \kL(E)$, if $y\leq\cl(A+x)$ and $y\not\nleq\cl(A)$, then $x\leq\cl(A+y)$.
\end{itemize}
If $\cl:\mathcal{L}(E)\to \mathcal{L}(E)$ satisfies the closure axioms (Cl1)-(Cl4), then  we call it a {\bf closure function}.
\end{definition}
\noindent
Similar to the matroid case for any $A\in\mL(E)$, we have that $r(A)=r(\cl(A))$.

\noindent
Finally, we define the restriction and the contraction operations for $q$-matroids; see \cite{byrne2022constructions,jurrius2018defining}.

\begin{definition}\label{def:restriction}
 Let $M:=(E,r)$ be a $q$-matroid and $A\leq E$ be any subspace of $E$.
 For every space $T\leq A$, we define $r_{M|_{A}}(T):=r(T)$. The $q$-matroid  $M|_{A}:=(A,r_{M|_{A}})$ is called the \textbf{restriction of $M$ to $A$}. Define a map
 $$r_{M/A} : \mathcal{L}(E/A) \to \mathbb{Z}, \ T\mapsto r(\pi^{-1}(T))-r(A),$$
where $\pi: E \to E/A$ is the canonical projection.
 Then the $q$-matroid $M/A :=(E/A,r_{M/A})$ is called the \textbf{contraction of $M$ from $A$}.
\end{definition}

We conclude with the notion of {\em dual matroid} \cite{jurrius2018defining}, which we will use in Sections~\ref{sec:cyclicflats} and~\ref{sec:rank-metric}.

\begin{definition}
Let $M=(E,r)$ be a $q$-matroid and consider the function
\begin{align*}
    r^\ast: \mathcal{L}(E)\to \Z, \ A \mapsto \dim(A)-r(E)+r(A^\perp).
\end{align*}
Then $r^\ast$ is a rank function and $M^*=(E,r^*)$ is a $q$-matroid, called the \textbf{dual $q$-matroid} of $M$.
\end{definition}


\section{Cyclic Spaces and Cyclic Flats}\label{sec:cyclicflats}
This section is devoted to the introduction of the $q$-analogue of cyclic flats and to present some of the properties of these objects. This will be the starting point for establishing a description of $q$-matroids in terms of cyclic flats.
For the remainder, $M=(E,r)$ will denote an arbitrary but fixed $q$-matroid with ground space $E$ and rank function~$r$.

\subsection{Cyclic Spaces}
We first define what it means for a space to be a cyclic subspace of a $q$-matroid.

\begin{definition}\label{def:cyclic}
  We say that $A\in\mL(E)$ is \textbf{cyclic} if $A=\langle 0 \rangle$ or $r(A)=r(B)$ for all~$B\in\Hyp(A)$.
\end{definition}

\textcolor{black}{Equivalently, a space $A \in\mL(E)$ is cyclic if $\cl(A)=\cl(B)$ for every $B \in \Hyp(A)$.}

\begin{example}
A trivial example of a cyclic space is given by a circuit of $M$. Indeed, if $C$ is a circuit and $D\in\Hyp(C)$, then $r(D)=\dim(D)=\dim(C)-1=r(C).$
\end{example}

\begin{definition}\label{def:cycclosed}
    For each $A \in \kL(E)$, define the set:
    \[\Cyc_r(A):=\{x \leq A\mid  r(B+x) = r(B) \; \textnormal{ for all }\; B \in\Hyp(A)\}.\]
    If it is clear from the context, we will write $\Cyc:=\Cyc_r$ for the $q$-matroid $M=(E,r)$.
\end{definition}

\begin{lemma}\label{lem:Cyc}
   Let $A \in\mL(E)$. Then
   \[
      \Cyc(A)= \mP\left( \sum_{x \in \Cyc(A)} x \right) = \mP\left( \bigcap_{H \in \Hyp(A)} \cl(H) \cap A\right).
   \]
\end{lemma}

\begin{proof}
  Let $x,y \in \Cyc(A)$ and let $z\leq x+y$. Let $D\in\Hyp(A)$. If $z\leq D$ then $r(z+D)=r(D)$.  Further, this will occur if both $x$ and $y$ are subspaces of $D$. Suppose that $z \nleq D$. Then we may assume that $x \nleq D$ and so~$r(z+D)=r(A)=r(x+D) = r(D)$.
  It follows that $z \in \Cyc(A)$ and so $\Cyc(A)= \mP\left( \sum_{x \in \Cyc(A)} x \right)$.
  To see that the second equality holds, note that for any $x \leq \mL(E)$, we have $x \in \Cyc(A)$ if and only if $x$ is contained $A$ and in the closure of every member of $\Hyp(A)$, which holds if and only if $x \leq \displaystyle{ \bigcap_{H \in \Hyp(A)} \cl(H) \cap A}$.
\end{proof}

Note that if $x$ is a loop contained in $A\in\mL(E)$, then $x\in\Cyc(A)$.

\begin{proposition}\label{prop:cyc}
    Let $A \in\mL(E)$. The following are equivalent.
    \begin{enumerate}
        \item[(1)] $A$ is cyclic.
        \item[(2)] For all $x \leq A$, we have that $r(x+B)=r(B)$ for all $B \in \Hyp(A)$.
        \item[(3)] $\Cyc(A)= \mP(A)$.
        \item[(4)] $\displaystyle A = \sum_{x \in \Cyc(A)} x$.
        \item[(5)] $\displaystyle{ A \leq \bigcap_{H \in \Hyp(A)} \cl(H)}$.
    \end{enumerate}
\end{proposition}
\begin{proof}
  The equivalence of (1) and (2) is immediate from the Definition \ref{def:cyclic} and the equivalence of (2) and (3) from Definition \ref{def:cycclosed}.
  If (3) holds then $\displaystyle A = \sum_{x \in \mP(A)}  x=\sum_{x \in \Cyc(A)} x$, and so we have (4). Conversely, if (4) holds, then we see that (3) holds by Lemma \ref{lem:Cyc}.
  That (1) and (5) are equivalent follows immdiately from Definition \ref{def:cyclic}.
\end{proof}

\begin{lemma}\cite[Lemma 3.2]{byrne2022constructions}\label{lem:r(B+x)}
   Let $A,B,x\in\mL(E)$, with $B \leq A$.
   If $r(B+x)=r(B)$ then $r(A+x)=r(A)$.
\end{lemma}

\begin{lemma}\label{lem:sumCicyclic}
   Let $C_1,C_2,\ldots,C_\ell$ be a collection of cyclic subspaces of $E$.
   Then $C_1+\cdots + C_\ell$ is also cyclic.
\end{lemma}

\begin{proof}
    Let $C=C_1+C_2$ and let $X\in\Hyp(C)$. 
    We claim that $r(C)=r(X)$. 
    Clearly, either $X \cap C_1=C_1$, or $X\cap C_1$ has codimension $1$ in $C_1$.
    Suppose the latter case, so there exists $y \leq C_1$ such that $y+X = C$. 
    Therefore, since $C_1$ is cyclic, we have that $r(y+(X \cap C_1))=r(X \cap C_1)$ and hence by Lemma \ref{lem:r(B+x)} we have $r(X)=r(y+X)=r(C)$. 
    On the other hand, if $X \cap C_1 = C_1$ then $X \cap C_2 \in \Hyp(C_2)$, since otherwise we arrive at the contradiction:
    $$C=C_1+C_2 = (X \cap C_1) + (X \cap C_2) \leq X \cap (C_1 + C_2) = X \cap C = X.$$
    As before, $X \cap C_2 \in \Hyp(C_2)$ implies that $r(C)=r(X)$.
    It follows that $C_1+C_2$ is a cyclic space. Now apply the same argument, iteratively, to arrive at the required result.
\end{proof}

Now consider the following operator.

\begin{definition}\label{def:cycop}
    The {\bf cyclic operator} of $M$ is the function
    	defined by
    	\begin{align*}
    	    \cyc_r:\mathcal{L}(E) \to\mathcal{L}(E), \ A \mapsto \cyc_r(A):=\sum_{\substack{C \leq A \\C \textnormal{ is cyclic}}} C.
    	\end{align*}
    If it is clear from the context, we will write $\cyc:=\cyc_r$ for the $q$-matroid $M$.
    We say that $A\in\mL(E)$ is {\bf cyclically closed}
    if
    $$\cyc(A) = A.$$
\end{definition}
\noindent From Lemma \ref{lem:sumCicyclic}, we see that $\cyc(A)$ is cyclic and is the unique maximal cyclic subspace of $A$.

One well-known construction of a $q$-matroid arises from the generator matrix of an $\F_{q^m}$-linear code; see \cite{jurrius2018defining, gorla2019rank}. Let $G$ be a $k \times n$ matrix over $\F_{q^m}$ and for every $U\in\mL(\F_q^n)$, let $A^U$ be a matrix whose columns form a basis of $U$. Then the map
\begin{equation}\label{eq:rank1}
    r:\mL(\F_q^n) \to \Z, \  U\mapsto \rk(G A^U),
\end{equation}
is the rank function of a $q$-matroid, which we denote by $M[G]$.

\begin{example}
Consider $\F_8=\F_{2^3}$ and let $\alpha\in\F_8$ be a primitive element satisfying $\alpha^3=\alpha+1$. Let $G$ be the following matrix with entries in $\F_8$,
$$ G:= \begin{pmatrix}
1 & \alpha &  \alpha^2 & 1 \\
0 & 1 & \alpha & 0 \\
\end{pmatrix}\in\F_8^{2\times 4}.$$
Let $r$ be the rank function of $M[G]$ and let $A_1:=\langle e_2,e_3\rangle\in\mL(\F_2^4)$. Then $r(A_1)=1$ and for all $B\in\Hyp(A_1)$, we have that $r(B)=1$ and hence 
$A_1$ is cyclically closed.
On the other hand, $A_2:=\langle e_1+e_2,e_3,e_4\rangle$ is not cyclically closed, indeed $\cyc(A_2)=\langle e_1+e_2+e_4,e_3 \rangle$.
\end{example}

We now list some basic properties of the cyclic operator.

\begin{theorem}\label{th:cyc123}
    For every $A,B\in\mL(E)$, the cyclic operator satisfies the following properties.
   \begin{itemize}[label={}]
       \item (cyc1)\label{pcyc1} $\cyc(A)\leq A$.
       \item (cyc2)\label{pcyc2} $A\leq B\Rightarrow\cyc(A)\leq \cyc(B)$.
       \item (cyc3)\label{pcyc3} $\cyc(\cyc(A))=\cyc(A)$.
    \end{itemize}
\end{theorem}
\begin{proof}
    It is immediate by Definition \ref{def:cycop} that (cyc1) holds.
    Let $A,B \in\mL(E)$.
    \textcolor{black}{From (cyc1), $\cyc(A) \leq A \leq B$ and hence $\cyc(A)$ is a cyclic subspace of $B$, which must therefore be contained in $\cyc(B)$, by definition of the cyclic closure.}
    From (cyc1) and (cyc2), we have that $\cyc(\cyc(A)) \leq \cyc(A)$.
    Since $\cyc(A)$ is itself cyclic and contained in $\cyc(A)$, we have
    $$\cyc(A) \leq \sum_{\substack{C \leq \cyc(A)\\
    C \textnormal{ is cyclic }}} C = \cyc(\cyc(A)).$$
    Therefore, (cyc3) holds.
\end{proof}

\begin{corollary}
     Let $A\in\mL(E)$ and let $\cyc(A) \leq X \leq A$. Then $\cyc(X)=\cyc(A)$.
\end{corollary}

\begin{remark}\label{lem:cyc_dep}
Every \textcolor{black}{nonzero} cyclic space  is dependent. Indeed, if \textcolor{black}{$A\ne \langle 0 \rangle$} is cyclic and independent, then all its subspaces have to be independent, so, clearly, $A$ cannot have the same rank as its hyperplanes.
\end{remark}

Lemma \ref{lem:sumCicyclic} shows that every open space (every sum of circuits) is cyclic. We will shortly see a converse. First, recall the following result. 

\begin{lemma}\cite[Proposition 86]{byrne2020cryptomorphisms}\label{lem:Fperpopenspace}
The flats of the $q$-matroid $M = (E, r)$ are the orthogonal spaces of the open spaces of the dual $q$-matroid $M^\ast$. That is, $F$ is a flat of $M$ if and only if $F^\perp$ is an open space of $M^\ast$.
\end{lemma}

In Proposition \ref{thm:cycequivalence} we characterize the cyclic spaces of $M$. In order to do this, for every integer $i$ consider the following set
$$ N_i := \{A \in\mL(E) \mid \nu(A)=i\},$$
i.e., the set of subspaces of $E$ with nullity equal to $i$.

The equivalence of (1) and (3) of the following proposition was shown in \cite[Lemma 13]{johnsen}. Note that in \cite[Definition 11]{johnsen}, a cyclic space is defined as being minimal with respect to inclusion in $N_a$ for some $a$.
\begin{proposition}\label{thm:cycequivalence}
 Let $A\in\mL(E)$ with $\nu(A)=a$. The following are equivalent.
 \begin{enumerate}[label={}]
     \item[(1)] $A$ is cyclic in $M$.
     \item[(2)] $A$ is minimal with respect to inclusion in $N_a$.
     \item[(3)] $A^\perp$ is a flat of the dual $q$-matroid $M^\ast$.
     \item[(4)] $A$ is the vector space sum of the circuits contained in $A$.
 \end{enumerate}
\end{proposition}

\begin{proof}
~{\bf }\\
\noindent
{$(1) \Leftrightarrow (2)$}:
Let $D\in\Hyp(A)$. Then
$\nu(D) = \dim(A)-1 - r(D) \geq \dim(A)-1 - r(A) = a-1.$
If $A$ is cyclic, then $r(D)=r(A)$ and so $\nu(D)=a-1$, which implies that $D \notin N_a$. Conversely, if $A$ is minimal in $N_a$, then $D \notin N_a$ and so $r(D)=r(A)$. Since this holds for arbitrary $D$ we obtain the equivalence of (1) and (2).
\\
\noindent
{$(1) \Leftrightarrow (3)$}:
$A^\perp$ is a flat in $M^*$ if and only if for all $x \nleq A^\perp$ we have
$r^*(A^\perp +x) = r^*(A^\perp) +1.$ For any $x \nleq A^\perp$, we have
\begin{eqnarray*}
   r^*(A^\perp +x) = r^*(A^\perp) +1 &\Leftrightarrow &
   \dim(A^\perp +x) + r(A \cap x^\perp) = \dim(A^\perp)+r(A)+1  \\
   & \Leftrightarrow & r(A\cap x^\perp) = r(A).
\end{eqnarray*}
In particular, since every hyperplane in $A$ has the form $A \cap x^\perp$ for $x \nleq A^\perp$, it follows that $A^\perp$ is a flat of $M^*$ if and only if $A$ is a cyclic space of $M$.\\
\textcolor{black}{{$(3) \Leftrightarrow (4)$}:
This follows directly from Lemma \ref{lem:Fperpopenspace}.}
\end{proof}

 Clearly, the properties of being a cyclic, a cyclically closed, or an open space of $M$ all coincide.

The following corollary immediately follows from Proposition \ref{thm:cycequivalence}.
\begin{corollary}\label{cor:loops}
$E$ is cyclic in $M$ if and only if the dual $q$-matroid $M^\ast$ does not contain a loop. In particular, $A\in\mL(E)$ is cyclic if and only if $(M|_A)^\ast$ does not contain a loop.
\end{corollary}

\begin{lemma}
   Let $A \in\mL(E)$ and let $X \leq A$ such that $A = \cyc(A)\oplus X$.
   Then $X$ is independent.
\end{lemma}

\begin{proof}
    Suppose that $X$ is not independent. Then $X$ contains a circuit $C$. Since $C$ is cyclic in $M$, by Lemma \ref{lem:sumCicyclic} $\cyc(A)+C$ is cyclic. Since $\cyc(A)$ is the unique maximal cyclic subspace of $A$ it follows that $C \leq \cyc(A)$, which yields a contradiction.
\end{proof}

The following result which will be crucial for establishing a $q$-cryptomorphism based on cyclic flats.

\begin{lemma}\label{lem:rA-rcycA}
    Let $A \in\mL(E)$. Then $r(A)-r(\cyc(A)) = \dim(A) - \dim(\cyc(A))$.
\end{lemma}
\begin{proof}
Let $\nu(A) = a$, thus $A \in N_a$. We claim that
$\cyc(A) \in N_a$.
    If $A$ is not cyclic, then there exists a subspace $Y\lneq A$, $Y\in N_a$ such that $Y$ is cyclic, by Proposition \ref{thm:cycequivalence}. From Theorem \ref{th:cyc123}, it follows that $Y=\cyc(Y)\leq \cyc(A)$ and so $a\leq \nu(\cyc(A))\leq a$, which implies that $\cyc(A)\in N_a$.
\end{proof}

\begin{proposition}\label{prop:Cyccyc}
   Let $A \in\mL(E)$. Then $\displaystyle \cyc(A)=\sum_{x \in \Cyc(A)} x$.
\end{proposition}

\begin{proof}
    The statement clearly holds if $A$ is cyclic, so assume that $\cyc(A) \lneq A$.
    Let $x \in \Cyc(A)$ such that $x \nleq \cyc(A)$.
    Let $H \in \Hyp(A)$ such that $\cyc(A) \leq H$ and $x \nleq H$.
    By the definition of $\Cyc(A)$, we have that $r(x+H)=r(H)$.
    By Lemma \ref{lem:rA-rcycA}, $A/\cyc(A)$ is independent in $M/\cyc(A)$.
    Therefore, $r(V)-r(\cyc(A)) = \dim(V)-\dim(\cyc(A))$ for every subspace $V$ satisfying
    $\cyc(A) \leq V \leq A$ and so we have that
    \[
        r(x+H)-r(H) = r(x+H)-r(\cyc(A))-r(H)+r(\cyc(A))=\dim(x+H)-\dim(H)=1,
    \]
    which yields a contradiction. It follows that $\Cyc(A) \subseteq \mP(\cyc(A)).$
    Conversely, since $\cyc(A)$ is cyclic, by Proposition \ref{prop:cyc}
    for any $x \leq \cyc(A)$ and any $H \in \Hyp(\cyc(A))$ we have that
    $r(x+H)=r(H)$ and so $r(x+H')=r(H')$ for any $H' \in \Hyp(A)$. It follows that $\mP(\cyc(A)) \subseteq \Cyc(A)$ and $\Cyc(\cyc(A))=\Cyc(A)$. The result now follows from Proposition \ref{prop:cyc}.
\end{proof}

The collection of cyclic spaces of $M$ forms a lattice, such that for every pair of cyclic spaces $C_1, C_2$, the join is defined by $C_1\vee C_2:=C_1+C_2$ and the meet is defined by $C_1\wedge C_2:=\cyc(C_1\cap C_2)$. Indeed,  by Lemma \ref{lem:sumCicyclic}, the sum of two cyclic spaces $C_1,C_2$ is cyclic. However, the intersection of a pair of cyclic spaces is not  cyclic in general: for example the intersection of two circuits is independent and hence not cyclic. In Example \ref{ex:first}, we provide a specific counterexample.

\begin{example}\label{ex:first}
Consider $\F_8=\F_{2^3}$ and let $\alpha\in\F_8$ be a primitive element satisfying $\alpha^3=\alpha+1$. Let $G$ be the following matrix with entries in $\F_8$,
$$ G:= \begin{pmatrix}
1 & \alpha & 1 & \alpha^2 & \alpha^4 \\
\alpha^3 & \alpha^4 & \alpha^4 & 1 & 1\\
\end{pmatrix}\in\F_8^{2\times 5}.$$
Let ${M}[G]=(\F_2^5,r)$ be the $q$-matroid associated to $G$.
With the aid of the computer algebra system \textsc{Magma} \cite{MR1484478}, we have checked that $M[G]$ contains $102$ cyclic spaces. Among these, consider for instance $U=\langle e_2, e_3, e_4, e_5 \rangle$ and $V=\langle e_1+e_4, e_2+e_5, e_3+e_5 \rangle$. We have that $V\cap U=\langle e_2+e_5, e_3+e_5 \rangle$  is independent and hence cannot be cyclic by Remark \ref{lem:cyc_dep}.

\end{example}

\subsection{Cyclic Flats}

In this subsection, we focus on \emph{cyclic flats}, which are simultaneously cyclic spaces and flats, i.e. spaces that are both open and closed in the $q$-matroid $M$. We show that also the collection of cyclic flats of a $q$-matroid forms a lattice and we prove that this lattice, together with the rank values of the cyclic flats, uniquely determines the $q$-matroid.
\begin{definition}
$F\in\mL(E)$ is a {\bf cyclic flat} if $\cyc(F)=F$ and $\cl(F)=F$. In terms of the rank function, a cyclic flat $F$ satisfies the following two properties:
    \begin{enumerate}
    \item $r(F+x)>r(F)$ for any $x\in\mL(E),$ such that $x\not\leq F$.
    \item $r(D)=r(F)$ for any $D \in\Hyp(F)$.
\end{enumerate}
We write $\mZ_r$ to denote the collection of cyclic flats of $M$. If it is clear from the context, we will simply write $\mZ$.
\end{definition}

The cyclic operator and the closure operator are closely related. Their interaction is also expressed by the following preliminary results.

\begin{lemma}\label{lem:cl(C)}
Let  $X\in\mL(E)$ be cyclically closed.
Then $\cl(X)\in \mZ$.
\end{lemma}
\begin{proof}
Since $\cl(X)$ is a flat, we need only to show that it is cyclic. Assume that $V\in\Hyp(\cl(X))$. If $X< V< \cl(X)$, then $\cl(X) = \cl(V)$ and in particular $r(V)=r(\cl(X))$. On the other hand, if $X$ is not contained in $V$, then $X\cap V\in\Hyp(X)$ and, since $X$ is cyclically closed, $r(X\cap V)=r(X)$. This shows that $r(X)=r(X\cap V)\leq r(V)$, which implies that $r(\cl(X))=r(V)$, hence $\cl(X)$ is cyclic.
\end{proof}

\begin{lemma}\label{lem:cyc(C)}
  Let $F\in\mL(E)$ be a flat of $M$.
  Then $\cyc(F) \in \mZ$.
\end{lemma}

\begin{proof}
    We need to show that for every flat $F$, $\cyc(F)$ is also a flat. Let $H=\Hyp(F)$.
    For any $A\in H$ and any $x \leq \cyc(F)$
    we have $r(x+A)=r(A)$, i.e. $x \leq \cl(A)$. It follows that $\cyc(F) \subseteq \bigcap_{A \in H} \cl(A).$
    On the other hand, if $x \leq \bigcap_{A \in H} \cl(A)$ then
    $x \leq F$ such that $r(x+A)=r(A)$ for every hyperplane $A$ in $F$. 
    It follows that $\cyc(F) = \bigcap_{A \in H} \cl(A)$, and in particular is an intersection of flats of $M$. Then $\cyc(F)$ is itself a flat.
\end{proof}

It follows immediately from Lemmas \ref{lem:cl(C)} and \ref{lem:cyc(C)} that for every subspace $X \leq E$, we have
\begin{align}\label{clcyc}
     \cl(\cyc(X)) \in\mZ \text{ and }\cyc(\cl(X)) \in\mZ.
\end{align}

Moreover, the following properties hold.

\begin{lemma}
Let $\cl^\ast$ and $\cyc^\ast$ denote the closure and cyclic operators with respect to the dual
matroid $M^*=(E,r^*)$. 
For every $A\in\mL(E)$ we have that
\begin{enumerate}
    \item[(1)] $\cyc^\ast(A)^\perp = \cl(A^\perp)$ and $\cyc(A)^\perp = \cl^\ast(A^\perp)$.
    \item[(2)] $\cl(\cyc(A))\cap A =\cyc(A)$.
    \item[(3)] $\cyc(\cl(A))+A = \cl(A)$.
\end{enumerate}
\end{lemma}
\begin{proof}
\begin{enumerate}
    \item[(1)] Since $\cyc(A)\leq A$, we have that $A^\perp \leq \cyc(A)^\perp$. From Proposition \ref{thm:cycequivalence}, we have that $\cyc(A)^\perp$ is a flat in $M^\ast$, hence
     $\cl^\ast(A^\perp) \leq \cyc(A)^\perp$.
   Moreover, as $A\leq \cl(A)$, then $\cl(A)^\perp\leq A^\perp$. Again from Proposition \ref{thm:cycequivalence}, we have that $\cl(A)^\perp$ is a cyclic space in $M^\ast$ and hence
   $\cl(A)^\perp = \cyc^\ast(\cl(A)^\perp)\leq \cyc^\ast(A^\perp)$.
   Therefore,
    \begin{align*}
        \cyc^\ast(A^\perp)^\perp \leq \cl(A) \leq \cyc^\ast(A^\perp)^\perp,
    \end{align*}
    showing that $\cyc^\ast(A)^\perp = \cl(A^\perp)$. By duality, we obtain that $\cyc(A)^\perp = \cl^\ast(A^\perp)$.

    \item[(2)] Clearly, $\cyc(A)\leq \cl(\cyc(A))\cap A  \leq \cl(\cyc(A))$.
    Therefore, $\cyc(A)$ and $\cl(\cyc(A))\cap A$ have the same rank and so $\nu(\cl(\cyc(A))\cap A)=a+s$, where $a=\nu(\cyc(A))$ and $s=\dim((\cl(\cyc(A))\cap A)/\cyc(A))$.
    Let $B=\cyc(\cl(\cyc(A))\cap A)$. By Theorem \ref{th:cyc123}, since
    $\cyc(A) \leq \cl(\cyc(A))\cap A \leq A$ then $\cyc(A) \leq B \leq \cyc(A)$ and so $B=\cyc(A)$. By Proposition \ref{thm:cycequivalence},
    $a=\nu(\cyc(A))=\nu(B)=a+s$ and hence $s=0$. The result now follows.

\item[(3)] By taking the orthogonal complements on both sides of Part (2) and applying Part (1) we get the desired result.

\end{enumerate}

\end{proof}

The next proposition shows that the collection of cyclic flats of a $q$-matroid forms a lattice under inclusion, which is not induced by the lattice of subspaces of the $q$-matroid nor the one of flats.

\begin{proposition}\label{prop:lattice}
   The set $\mZ$ of cyclic flats of a $q$-matroid forms a lattice under inclusion, where for any two cyclic flats $F_1, F_2$ the meet is defined by $F_1\wedge F_2:=\cyc(F_1\cap F_2)$ and the join is defined as $F_1\vee F_2:=\cl(F_1 + F_2)$.
\end{proposition}
\begin{proof}
 $F_1\cap F_1$  is a flat and $\cyc(F_1\cap F_2)$ is a cyclic flat by Lemma \ref{lem:cyc(C)}. If $F$ is a cyclic flat contained in $F_1 \cap F_2$ then by Theorem \ref{th:cyc123}, (cyc2), we have that $F \leq \cyc(F_1\cap F_2)$. By Lemma \ref{lem:sumCicyclic} and Lemma \ref{lem:cl(C)}, it immediately follows that  $\cl(F_1 + F_2)$ is a cyclic flat.
 If $F \in \mZ$ such that $F_1,F_2 \leq F$, then
 $F_1+F_2 \leq F$ and so $\cl(F_1 +F_2) \leq F$, by (Cl2) and (Cl3).
 Finally note that the flat $\cl(\langle 0\rangle)$ is cyclic and it is the unique minimal element of $\mZ$ while $\cyc(E)$ is the unique maximal element of $\mZ$.
\end{proof}

Combining Proposition \ref{prop:lattice} with Lemma \ref{lem:rA-rcycA}, we get that for every pair of cyclic flats $X,Y \in\mZ$,
\begin{align*}
    r(X\vee Y)&=r(\cl(X+ Y)) = r(X+ Y),\\
    r(X\wedge Y)&=r(\cyc(X\cap Y))=r (X\cap Y) - \dim((X\cap Y)/(X\wedge Y)).
\end{align*}

Brylowski outlined in \cite[Proposition 2.1]{brylawski1975affine} an algorithm for constructing the lattice of flats of a matroid from its lattice of cyclic flats, along with the ranks of the cyclic flats. In \cite[Section 5]{freij2021cyclic}, the authors also showed how to reconstruct the lattice of flats from the lattice of cyclic flats, along with their ranks.
The same construction given in \cite{freij2021cyclic} applies in the $q$-analogue and we give a brief sketch. For each $X\in\mL(E)$, define two cyclic flats $$X^\vee:=\bigvee_{\substack{Z:Z \in {\mathcal Z} ,\\ Z \leq X}} Z = \cl\left(\sum_{\substack{Z:Z \in {\mathcal Z} ,\\ Z \leq X}} Z\right) \leq \cl(X)  \ \  \text{ and } \ \  X^\wedge:=\bigwedge _{\substack{Z:Z \in {\mathcal Z} ,\\ X \leq Z}} Z =\cyc\left(\bigcap_{\substack{Z:Z \in {\mathcal Z} ,\\ X \leq Z}} Z\right),$$
where $\vee$ and $\wedge$ denote the join and meet, respectively of the lattice $\mZ$ and where the intersection of an empty set of spaces is equal to the whole space $E$.

The following Lemma for matroids can be read in \cite{freij2021cyclic}. We include a proof of the $q$-analogue for the convenience of the reader.

\begin{lemma}\label{lem:lattice}
Let $A\in\mL(E)$ and $Z\in\mL(E)$ be a cyclic flat of the $q$-matroid $M$ satisfying $r(Z) + \dim(A/(A\cap Z)) = r(A)$. Then $\cl(\cyc(A))\leq Z \leq \cyc(\cl(A))$.
\begin{proof}
For every $A,Z\in\mL(E)$ we have that
$$r(A)-r(A\cap Z)=r_{M/(A\cap Z)}(A/(A\cap Z)) \leq \dim(A/(A\cap Z)).$$
Hence
$$ r(A)\leq r(A\cap Z) + \dim(A/(A\cap Z))\leq r(Z)+\dim(A/(A\cap Z)).$$
\emph{Claim 1:} $r(A)= r(A\cap Z) + \dim(A/(A\cap Z))$ if and only if $\cyc(A)\leq A\cap Z$.
Indeed, having the equality means that $\nu(\cyc(A\cap Z)=\nu(A\cap Z) = \nu(A)=\nu(\cyc(A))$, where the first and last equalities follow from Lemma \ref{lem:rA-rcycA}. Since $\cyc(A)$ and $\cyc(A\cap Z)$ are minimal with respect to inclusion in $N_{\nu(A)}$ and $\cyc(A\cap Z)\leq \cyc(A)$, we must have that $\cyc(A)= \cyc(A\cap Z)\leq A\cap Z$. Conversely, if $\cyc(A)\leq A\cap Z$, then we have that $\cyc(A)=\cyc(A\cap Z)$, hence, by Lemma \ref{lem:rA-rcycA}, we have that $\nu(A)=\nu(A\cap Z)$, from which the claim follows.

\noindent\emph{Claim 2:} $r(A\cap Z) + \dim(A/(A\cap Z))= r(Z)+\dim(A/(A\cap Z))$ only if $Z\leq \cl(A)$. Assume that $Z\not\leq \cl(A)$. Let $z\leq Z$, such that $z\not\leq\cl(A)$. Then $r(A)<r(A+z)$ and by submodularity we have that
$$ r(A\cap Z) < r((A\cap Z)+z)\leq r(Z).$$
The two claims show that for any $A,Z$ satisfying $r(Z) + \dim(A/(A\cap Z)) = r(A)$, we must have $\cyc(A)\leq Z\leq \cl(A)$. So, $\cl(\cyc(A))\leq\cl(Z)$ and $\cyc(Z)\leq\cyc(\cl(A))$. If $Z$ is a cyclic flat, then $Z = \cl(Z) = \cyc(Z)$, and it follows that $\cl(\cyc(A))\leq Z \leq \cyc(\cl(A))$.
\end{proof}
\end{lemma}

Clearly, if $X$ is a cyclic flat then $X^\vee =X$. It can be shown, using an argument identical to that of \cite[Proposition 6, (i) $\Leftrightarrow$ (iii)]{freij2021cyclic}, which depends on Lemma \ref{lem:lattice}, that
if $X^\vee \leq X$, then $X$ is a flat if and only if for every cyclic flat $A$ satisfying  $X^\vee \lneq A \leq X^\wedge$,
\begin{align}\label{eq:checkflats}
    \dim(X \cap A)-r(A)<\nu(X^\vee).
\end{align}
This property can be checked if $\mZ$ and the ranks of its elements are known. Note that if this is the case, that is if $X$ is a flat, then $X^\vee = \cyc(X)$.

\textcolor{black}{
Hence, we draw the following conclusion.
\begin{proposition}\label{prop:reconstruction}
Every $q$-matroid $M$ is uniquely determined by its lattice of cyclic flats along with their rank values.
\end{proposition}
\begin{proof}
    By the algorithm outlined above, the lattice of cyclic flats along with their ranks uniquely determines the lattice of flats of $M$. By \cite{byrne2022constructions}, $M$ is uniquely determined by its lattice of flats.
\end{proof}
}
The next example illustrates how the reconstruction algorithm works.

\begin{example}\label{ex:2x4}
Consider the finite field $\F_{2^3}=\F_2[\alpha]$, where $\alpha^3=\alpha+1$. Let  $G\in\F_{2^3}^{2\times 4}$ be the  matrix
$$ G =\begin{pmatrix}
1 & 0 & 0 & 0\\
0 & 1 & \alpha & \alpha^2
\end{pmatrix}.$$

\noindent
Let $M[G]=(\F_2^4,r)$ be the $q$-matroid associated to $G$.
The only cyclic flat except for $\langle 0\rangle$ is $\langle e_2, e_3, e_4\rangle $; see Figure \ref{fig:firstlattice}. We have
{$\cyc(E) = \langle e_2,e_3,e_4  \rangle$}, which means that $(M[G])^\ast$ has a loop, by Corollary \ref{cor:loops}.

\begin{figure}[ht!]
\centering
\resizebox{0.14\textwidth}{!}{%
\begin{tikzpicture}
\node[shape=rectangle, draw=black] (00) at (0,0) {\tiny $\langle 0 \rangle $};
 \node[shape=rectangle,draw=black, fill=orange!30] (20) at (0,1) {\tiny $\langle e_2,e_3,e_4 \rangle$};
 \path [- ] (00) edge (20);
\end{tikzpicture}
}
\caption{Lattice of cyclic flats of the matroid ${M}[G]$ from Example \ref{ex:2x4}.}
\label{fig:firstlattice}
\end{figure}
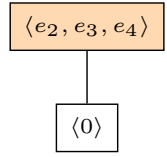

\noindent
We may apply the reconstruction result from the above discussion to obtain all the flats of $M[G]$. Take for instance the space $F=\langle e_1\rangle$. Then $F^\vee = \langle 0\rangle\leq F$ and $F^\wedge = \cyc(E) =\langle e_2,e_3, e_4\rangle $. The only cyclic flat $A$, satisfying $F^\vee \lneq A \leq F^\wedge$ is $\cyc(E)$. We can then verify that \eqref{eq:checkflats} is satisfied, i.e.
$$ 0-1 = \dim(\langle e_1\rangle \cap \langle e_2,e_3,e_4\rangle) - r(\langle e_2,e_3,e_4\rangle) < \nu(\langle 0\rangle )=0.$$
Hence, we conclude that $F=\langle e_1 \rangle$ is a flat and $\langle 0\rangle=F^\vee=\cyc(F)$. Moreover, by applying Lemma \ref{lem:rA-rcycA}, we have that $r(F)=1$.

\noindent
For another example, let $B=\langle e_2,e_3\rangle$. Then, as above, $B^\vee = \langle 0\rangle $ and $B^\wedge = \cyc(E) =\langle e_2,e_3, e_4\rangle$. The only cyclic flat $A$, satisfying $B^\vee \lneq A \leq B^\wedge$ is $\cyc(E)$. However, this time we have that
$$1=2-1 =  \dim(\langle e_2,e_3 \rangle \cap \langle e_2,e_3,e_4\rangle) - r(\langle e_2,e_3,e_4\rangle) > \nu(\langle 0\rangle)=0,$$
hence, $B$ is not a flat. With this procedure we can reconstruct the flats of $M[G]$.  We list them below, together with their ranks.
\begin{align*}
    \textnormal{rank } r=0: \quad & \langle 0\rangle . \\
    \textnormal{rank } r=1:  \quad  & \langle e_1\rangle, \langle e_1 + e_2\rangle,\langle e_1 + e_3\rangle, \langle e_1 + e_4\rangle, \langle e_1 + e_2 +e_3\rangle,\\
    \quad &\langle e_1 + e_2 +e_4\rangle, \langle e_1 + e_3 +e_4\rangle,\langle e_1+e_2+e_3 +e_4\rangle,\langle e_2,e_3,e_4\rangle. \\
     \textnormal{rank } r=2: \quad  & E.
\end{align*}

\end{example}

We conclude this section by providing a characterization in terms of cyclic flats of a well-known family of $q$-matroids, namely the family of \emph{uniform $q$-matroids} (c.f \cite{jurrius2018defining}). To this end, we denote by $0_{\mZ}:=\cl(\langle 0\rangle)$ the minimal element and by $1_{\mZ}:=\cyc(E)$ the maximal element of the lattice of cyclic flats of $M$.

The following result characterizes the independent spaces and the circuits of a $q$-matroid in terms of its cyclic flats.

\begin{lemma}\label{lem:ind-circuit-cflats}
Let $M=(E,r)$ be a $q$-matroid, then the following hold.
\begin{enumerate}
    \item $I\in\mL(E)$ is independent if and only if for every cyclic flat $X$, $\dim(I\cap X) \leq r(X)$.
    \item $C\in\mL(E)$ is a circuit if and only if $C$ is a minimal space such that there exists a cyclic flat $X$ satisfying $C\leq X$ and $\dim(C)=r(X)+1$.
\end{enumerate}
\end{lemma}
\begin{proof}
\begin{enumerate}
    \item  $I$ is independent if and only if $r(I)=\dim(I)$, in which case every subspace of $I$ is independent. In particular, if $I$ is independent then $r(I\cap X) =\dim(I\cap X)$ for every cyclic flat $X$. Then $\dim(I\cap X)\leq r(X)$ by (R2).
    Assume now that $I$ is not independent. We will construct a \textcolor{black}{cyclic flat} $X$ such that $r(X)<\dim(I \cap X)$. There exists a circuit $C\leq I$. Let $X:=\cl(C)$, which is a cyclic flat by Lemma \ref{lem:cl(C)}. Clearly, $C \leq I \cap X$ and so
    $$ X = \cl(C) \leq \cl(I \cap X) \leq X, $$
    which implies that $\cl(I\cap X)=X$. In particular, $r(X)=r(I\cap X)\leq \dim(I\cap X)$. Furthermore, as $I\cap X$ contains the circuit $C$, we have that $r(I\cap X) < \dim(I \cap X)$.
    \item A circuit $C$ is a minimal dependent space with $r(C)=\dim(C)-1$. Moreover, by Lemma \ref{lem:cl(C)}, the closure $\cl(C)$ is a cyclic flat with the property that $r(\cl(C))=r(C)$. Hence, by setting $X=\cl(C)$, we get the statement. \qedhere
\end{enumerate}
\end{proof}

\begin{definition}
Let $1\leq k\leq n$. For each $U\in\mL(E)$, define $r(U) :=\min\{k,\dim(U)\}$. Then $(E,r)$ is a $q$-matroid. It is called the \textbf{uniform $q$-matroid on $E$ of rank $k$} and is denoted by~$U_{k,n}$.
\end{definition}

\begin{proposition}\label{prop:uniform-cyclic}
   Let $M=(E,r)$ be a $q$-matroid \textcolor{black}{of rank $k$}, with $0<k<n$. Then the following are equivalent.
   \begin{enumerate}
       \item[(1)] $M=U_{k,n}$.
       \item[(2)] $\mZ_r$ contains only two elements, $0_{\mZ_r}=\langle 0 \rangle$ and  $1_{\mZ_r}=E$. Moreover $r(1_{\mZ_r})=k$.
   \end{enumerate}
\end{proposition}
\begin{proof}

That (1) implies (2) is immediate from the definition of the uniform $q$-matroid: the only proper flats of $U_{k,n}$ are the subspaces of $E$ of dimension less than $k$, which are all independent \textcolor{black}{and hence cannot be cyclic by Remark \ref{lem:cyc_dep}}. Then $0_{\mZ_r}=\langle 0\rangle$ and $1_{\mZ_r}=\cyc(E) = E$, the vector space sum of its subspaces of dimension $k+1$.

Suppose now that (2) holds and assume that $M$ is not the uniform $q$-matroid. Then there is a subspace $I\leq E$ of dimension $k$ that is not independent. Since $I$ is not independent, it contains a circuit $C$, with $\dim(C)\leq\: k$. By Lemma \ref{lem:ind-circuit-cflats}, there exists a cyclic flat $X$, such that $\langle 0 \rangle \lneq C\leq X$ and $\dim(C)=r(X)+1$, which implies that $r(X)\leq k-1$.
In particular, there exists a cyclic flat $X \notin \{0_\mZ, 1_{\mZ}\}$, which
contradicts (2). Therefore, $M=U_{k,n}$.
\end{proof}


\section{The Rank Function and Cyclic Flats}\label{sec:cryptomorphism}
In this section, we propose a $q$-cryptomorphism based on cyclic flats. To this end, we consider \textcolor{black}{a $q$-matroid $(E,r)$,} a lattice $\mZ$ of subspaces of $E$ and a function $r_{\mZ}$ on $\mL(E)$, which satisfies the axioms (R1)-(R3) of a rank function. Next, we show that $r(F)=r_{\mZ}(F)$ for every $F \in \mZ$ and we observe that the spaces in $\mZ$ are exactly the cyclic flats of the $q$-matroid $(E,r_{\mZ})$. This section is inspired by Sims' work \cite{sims1977extension}, who proved that any finite lattice is isomorphic to the lattice of cyclic flats of a finite matroid. The same result for matroids was also independently proved by Bonin and de Mier in \cite{bonin2008lattice}, with a different approach.

\begin{definition}
    Let $\mZ$ be a collection of subspaces of $E$ and suppose that
    $(\mZ,\leq,\vee ,\wedge)$ is a lattice with join and meet operations $\vee$ and $\wedge$, \textcolor{black}{ such that for every $Z_1,Z_2\in \mZ$, we have that $Z_1+Z_2\leq Z_1\vee Z_2$ and $Z_1\wedge Z_2\leq Z_1\cap Z_2$, respectively}. Let
    $f:\mZ \rightarrow \Z$ be a map.
    We define the following {\bf cyclic flat} axioms.
    \begin{description}
    \item[(Z1)] We have that $f(0_\mZ)=0$, where $0_\mZ$ is the minimal element of $\mZ$.
    \item[(Z2)] For every $F,G\in\mZ$ such that $G\lneq F$, we have:
    \begin{align*}\label{eq:differenceranks}
        0<f(F)-f(G)<\dim(F) -\dim(G).
    \end{align*}
    \item[(Z3)] For every $F,G\in\mZ$ we have:
    \begin{equation*} \label{eq:submodularity2}f(F)+f(G)\geq f(F \vee G) + f(F\wedge G) + \dim((F\cap G)/(F\wedge G)). \end{equation*}
\end{description}
If $(\mZ,f)$ satisfies the cyclic flat axioms, we say that $\mZ$ is a lattice of {\bf cyclic flats} with respect to $f$.
\end{definition}

The following preliminary result will be used to show that the lattice of cyclic flats of a $q$-matroid satisfies (Z1)--(Z3).

\begin{lemma}\label{lem:quotientflats}
   Let $F$ be a cyclic flat of $M$ and let $G<F$. Then $F/G$ is a cyclic flat of the $q$-matroid $M/G = (E/G, r_{M/G})$.
\end{lemma}
 \begin{proof}
 Let $z\leq E$ such that 
 $z\not\leq F$. Since $F$ is a flat of $M$, we have
 $$r_{M/G}((z+G)/G+F/G)= r_{M/G}((z+F)/G) =r(z+F)-r(G) > r(F)-r(G) = r_{M/G}(F/G),$$
 and hence $F/G$ is a flat of $M/G$.

 It remains to show that $F/G$ is cyclic. Every hyperplane in $F/G$ can be written as $D/G$, where $D\in\Hyp(F)$ such that $D$ contains $G$. Since $F$ is cyclic, we have:
 $$ r_{M/G}(D/G) = r(D) - r(G) = r(F)-r(G) = r_{M/G}(F/G),$$
 i.e. $F/G$ is cyclic in $M/G$.
 \end{proof}

\begin{theorem}\label{def:axiomscycflats}
Let $\mZ$ be the lattice of cyclic flats of $M$ and let $f:\mZ\to \Z$, be the map defined by $f(F)=r(F)$ for all $F\in\mZ$. Then $(\mZ,f)$ is a lattice of cyclic flats with respect to $f$, i.e. $(\mZ,f)$ satisfies (Z1)--(Z3).
\end{theorem}
\begin{proof}
(Z1) holds because $0_\mZ=\cl(\langle 0\rangle)$, i.e. it is the vector space sum of the loops of $E$.

To show that (Z2) holds, assume that $F$ and $G$ are two cyclic flats with $G<F$.  By the definition of the rank function $r_{M/G}$, we have:
$$ f(F)-f(G) = r(F)-r(G)=r_{M/G}(F/G)\leq \dim(F/G)=\dim(F)-\dim(G).$$
By Lemma \ref{lem:quotientflats}, $F/G$ is a cyclic flat in $M/G$ and by Remark \ref{lem:cyc_dep}, it must be dependent, forcing the inequality to be strict.

(Z3) follows from Lemma \ref{lem:rA-rcycA} applied to $F\cap G$, combined with submodultarity:
\begin{align*}
   \textcolor{black}{f(F) + f(G) } &= r(F) + r(G) \geq r(F+G) +r(F\cap G) \\
    &= r(F \vee G) + r(F\wedge G) + \dim((F\cap G)/(F\wedge G)) \\
    &=\textcolor{black}{f(F \vee G)+ f(F\wedge G) + \dim((F\cap G)/(F\wedge G))}. \qedhere
\end{align*}
\end{proof}

As an immediate consequence of Theorem \ref{def:axiomscycflats}, we have the following.

\begin{corollary}\label{cor:general}
 Let $F,G$ be \textcolor{black}{two distinct} cyclic flats of $M$, then
       \begin{align}\label{eq:general}
            r(F)-r(G)<\dim(F)-\dim(F\cap G).
       \end{align}
\end{corollary}
\begin{proof}
Assume that $F\nleq G$. Clearly $r(F\cap  G)\leq r(G)$ and so by (R2), we have $r(F)-r(G) \leq r(F)-r(F\cap G)$.
By Lemma \ref{lem:quotientflats}, $F/(F\cap G)$ is a cyclic flat in $M/(F\cap G)$, and is therefore dependent, as observed in Remark \ref{lem:cyc_dep}. It follows that
$$r(F)-r(G)\leq r_{M/(F\cap G)}(F/(F\cap G))<\dim(F/(F\cap G))=\dim(F)-\dim(F\cap G).$$
If $F\leq G$, then $F\cap G=F$ and $r(F)<r(G)$, since $F$ is a flat. Hence we have that
\begin{equation*}
      r(F)-r(G)< 0 = \dim(F) - \dim(F\cap G).                      \qedhere
\end{equation*}
\end{proof}

We introduce the following function.

\begin{definition}\label{def:rhoZ}
 Let $\mZ$ be a collection of subspaces of $E$. For every map $h : \mL(E) \longrightarrow \N_0$,
 we define $h_\mZ : \mL(E) \longrightarrow \N_0$ to be the function:
 \begin{align}\label{eq:convolution}
     h_{\mZ}(A) := \min\{h(F) + \dim((A+F)/F) \mid F\in\mZ\}, \textnormal{ for all } A \in \mL(E).
 \end{align}
\end{definition}

The remainder of this section will be devoted to proving \textcolor{black}{that if $\mZ$ is the lattice of cyclic flats of the $q$-matroid $(E,r)$, then $(E,r_\mZ)$ is a $q$-matroid, whose lattice of cyclic flats is equal to $\mZ$ and for which $r_\mZ=r$ as functions on $\mL(E)$.}

We will need to first prove a number of preliminary results.

\begin{lemma}\label{lem:extraaxiom}
Let $(\mZ,f)$ be a lattice of cyclic flats. Then, for every $F,G\in\mZ$, we have
\begin{equation}\label{eq:extraZ4}
    f(F\vee G) \leq f(G) + \dim(F/(F\cap G)).
\end{equation}
\end{lemma}
\begin{proof}
We distinguish three cases.
If $F\leq G$, then $F\vee G=G$ and Equation \eqref{eq:extraZ4} clearly holds.

\noindent
If $G\leq F$, then the result directly follows from (Z2).

\noindent
Finally, assume that $F\nleq G$ and $G\nleq F$. Then apply (Z2) to $F\wedge G$ and $F$ and apply (Z3) to $F$ and $G$ to obtain:
\begin{gather*}
    f(F) \leq f(F\wedge G) + \dim(F/(F\wedge G))\\
   f(F\vee G) + f(F\wedge G) \leq  f(F)+f(G)- \dim((F\cap G)/(F\wedge G)).
\end{gather*}
Combining these inequalities we get the desired result.
\end{proof}

 It is straightforward to check that if  $(\mZ,f)$ is a lattice of cyclic flats of $E$ then $f$ and $f_\mZ$ both agree on $\mZ$, as we show in the following proposition.

\begin{proposition}\label{prop:restrictionrho}
Let $(\mZ,f)$ be a lattice of cyclic flats. Then, for every $A\in \mZ$, we have $f_\mZ(A) =f(A)$.
\end{proposition}
\begin{proof}
Let $A\in \mZ$. By definition,
$f_\mZ(A)=\min\{f(F) + \dim((A+F)/F)\mid F\in \mZ\}$. In particular,
$$f_\mZ(A) \leq f(A) + \dim(A/A)=f(A).$$
For the opposite inequality, by applying Lemma \ref{lem:extraaxiom}, we have:
$$f(A)\leq f(A\vee G) \leq f(G) + \dim((A+G)/G),$$
for every $G\in\mZ$. In particular, $f(A)\leq f_\mZ(A)$.
\end{proof}

The following inequalities will be very useful for the next propositions.
\begin{lemma}\label{lem:dims}
Let $X,Y, V, W \leq E$. The following inequalities hold:
\begin{eqnarray}
    \label{eq:rleqdim}
    \dim(X/(X\cap V))\leq   \dim(X/(X\cap Y)) + \dim(Y/(Y\cap V)) .\\
    \dim((X+Y+V+W)/(V+W)) +\dim(((X\cap Y)+(V\cap W))/(V\cap W)) \nonumber
   \\\leq \dim((X+V)/V) + \dim((Y+W)/W). \label{eq:third}
\end{eqnarray}
\end{lemma}
\begin{proof}
The inequality (\ref{eq:rleqdim}) holds if and only if
\[
 \dim(X \cap V) + \dim(Y) -(\dim(X \cap Y) + \dim(Y \cap V)) \geq 0.
\]
However, this holds since
\[
  \dim(X \cap Y) + \dim(Y \cap V) = \dim(X \cap Y + Y \cap V)-\dim(X \cap Y \cap V) \leq \dim(Y) - \dim(X \cap Y \cap V).
\]

To see that (\ref{eq:third}) holds, we note the following:
\begin{align*}
   \dim((X+V)/V) + \dim((Y+W)/W) = \\
   \dim(X+Y + V+ W) + \dim((X+V) \cap(Y+W)) - \dim(V+W) -\dim(V \cap W)=\\
   \dim((X+Y + V+ W)/(V+W))+ \dim((X+V) \cap(Y+W)) -\dim(V \cap W)\geq \\
   \dim((X+Y + V+ W)/(V+W))+\dim(((X\cap Y)+(V\cap W))/(V\cap W)),
 \end{align*}
 where the last inequality holds since $(X\cap Y)+(V\cap W)\leq (X+V)\cap(Y+W)$.
\end{proof}

\begin{proposition}\label{prop:rzsubmod}
\textcolor{black}{Let $(\mZ,f)$ be a lattice of cyclic flats.}
Then $f_{\mZ}$ satisfies the axioms~(R1)--(R3). In particular, $(E,f_{\mZ})$ is a $q$-matroid.
\end{proposition}
\begin{proof}
\begin{itemize}
\item[(R1)] For every $A\in\mL(E)$, we have $f_\mZ(A)\geq 0$ being the sum of two non-negative integers. Moreover, since $f(0_\mZ)=0$, by (Z1), we have that
\begin{align*}
    f_\mZ(A) \leq f(0_\mZ)+ \dim(A/(A\cap 0_{\mZ})) \leq\dim A.
\end{align*}
\item[(R2)] Let $A\leq B$ \textcolor{black}{and let $F\in\mZ$ be such that $f_\mZ(B)=f(F)+\dim(B/(B\cap F))$}. Then (R2) follows by the definition of $f_\mZ$, since, $$f_\mZ(B)=f(F)+\dim(B/(B\cap F))\stackrel{\eqref{eq:rleqdim}}\geq f(F)+\dim(A/(A\cap F))\geq f_\mZ(A).$$

\item[(R3)] Let $A,B\leq E$ and $F_A,F_B\in\mZ$ be such that
$f_{\mZ}(A) = f(F_A)+\dim(A/A \cap F_A)$ and $f_{\mZ}(B) = f(F_B)+\dim(B/B \cap F_B)$. \textcolor{black}{Using the fact that $F_A+F_B\leq F_A \vee F_B$ and $F_A \wedge F_B \leq F_A \cap F_B$,} we have that
\textcolor{black}{{\small{
\begin{align*}
    f_{\mZ}(A+B)+f_{\mZ}(A\cap B)&\leq f(F_A\vee F_B) + \dim((A+B+(F_A\vee F_B))/(F_A\vee F_B)) +\\ &\ \ \ \ f(F_A\wedge F_B) + \dim(((A\cap B)+(F_A\wedge F_B))/(F_A\wedge F_B)) \\
     &\stackrel{(Z3)}\leq f(F_A) + f(F_B) - \dim((F_A \cap F_B)/(F_A \wedge F_B)) +\\ &\ \ \ \ \ \dim((A+B+F_A+ F_B)/(F_A+ F_B))+ \\ &\ \ \ \ \  \dim(((A\cap B)+(F_A\cap F_B))/(F_A\wedge F_B)) \\
    &\stackrel{\eqref{eq:third}}\leq f(F_A)+f(F_B) + \dim((A+F_A)/F_A) +  \dim((B+F_B)/F_B) \\
    &=f_{\mZ}(A) + f_{\mZ}(B).
\end{align*}}}}
This establishes the submodularity of $f_\mZ$.
\end{itemize}
Since (R1), (R2) and (R3) hold, we conclude that $(E,f_\mZ)$ is a $q$-matroid.
\end{proof}

\begin{remark}\label{rem:convolution_submodula_modular}
A function that satisfies (R3) with equality is called \textbf{modular}. Let $S$ be a set and let $2^S$  be the collection of subsets of $S$.
In \cite[Theorem 2.5]{lovasz1983submodular}, Lov\'asz showed that if $f,g$ are two functions defined on $2^S$ such that $f$ is submodular and $g$ is modular, the convolution defined as
$$ f\ast g (B) = \min\{f(A) +g(B-A) \mid A\subseteq B\}$$
is submodular.
It is straightforward to check that the same can be said when the two functions are defined on $\mL(E)$, and $g(A)=\dim(A)$ for all $A\in\mL(E)$; see \cite[Theorem 24]{ceria2021direct}. In Proposition \ref{prop:rzsubmod}, we showed that submodularity is also satisfied when the convolution is not taken over all the spaces, as in \eqref{eq:convolution}.
Thus if $f$ is submodular on $\mL(E)$, then $f_\mZ$ is obtained as a convolution of $f$ with the dimension function and so inherits the submodularity property from $f$.
\end{remark}

\begin{theorem}\label{thm:proofthm}
\textcolor{black}{Let $(\mZ,f)$ be a lattice of cyclic flats. Then $\mZ$ is the lattice of cyclic flats of the $q$-matroid $M_\mZ:=(E,f_\mZ)$. }
\end{theorem}
\begin{proof}
From Proposition \ref{prop:rzsubmod}, we have that $M_\mZ$ is a $q$-matroid.
Let $F\in\mZ$. We show that $F$ is a flat of $M_{\mZ}$. Let $x\in\mL(E)$ be such that $x\nleq F$. Then there exists $G\in\mZ$ such that
\begin{align*}
    f_{\mZ}(F+x) &= f(G) + \dim((F+x)+G)/G).
\end{align*}
If $F=G$, then  $f_{\mZ}(F+x)=f(F)+1=f_\mZ(F)+1$.

\noindent
If $F\ne G$, then $F < F \vee G$ and so by (Z2) and Lemma \ref{lem:extraaxiom}, we have that
\[
f(F)-f(G) < f(F \vee G) - f(G)\leq \dim(F+G) -\dim(G).
\]
Therefore,
\begin{align*}
    f_{\mZ}(F+x) &= f(G) + \dim((x+F+G)/G) \\
    & > f(F) -\dim(F+G) +\dim(G) + \dim(x+F+G) -\dim(G) \\
    & = f_\mZ(F) -\dim(F+G) + \dim(F+G)+1 -\dim(x \cap (F+G))\\
    & = f_\mZ(F) +1 -\dim(x \cap (F+G))\\
    &=f_\mZ(F).
\end{align*}
In both cases we see that $f_\mZ(F+x)>f_{\mZ}(F)$, so $F$ is a flat of $M_\mZ$.

\noindent
Now, let $D\in \Hyp(F)$ and let $G\in\mZ$ be such that
$$f_\mZ(D) = f(G) + \dim((D+G)/G).$$
If $F=G$, then $f_\mZ(D) = f(F) = f_\mZ(F)$, so $F$ is cyclically closed in $M_\mZ$.

\noindent
If $F\ne G$, then again by (Z2) and Lemma \ref{lem:extraaxiom}, we have that
\begin{align*}
    f_\mZ(D) &= f(G) + \dim((D+G)/G) \\
    &
    > f(F)-\dim((F+G)/G) + \dim((D+G)/G) \\
    & = f_\mZ(F) - \dim(F+G)+\dim(D+G) \\
    & \geq f_\mZ(F)-1.
\end{align*}
We conclude that $f_\mZ(D)=f_\mZ(F)$ and hence $F$ is cyclic in $M_{\mZ}$.

Finally, we show that every cyclic flat of $M_\mZ$ is in $\mZ$. Let $F\in\mL(E)$ be a cyclic flat. Let $G\in \mZ$ be such that
$ f_\mZ(F) = f(G)+\dim((F+G)/G).$ Since $F$ is cyclic, for every subspace $D\in\Hyp(F)$, we have
$$ f_\mZ(D) = f_\mZ(F) \geq f(G)+\dim((D+G)/G)\geq f_\mZ(D).$$
This implies that $F\leq G$. Since $F$ is also a flat, for every $x\nleq F$, we have that
$$f_\mZ(F)<f_\mZ(F+x) \leq f(G)+\dim((F+x+G)/G),$$
which implies that $x\nleq G$. These together show that $F=G$ and, in particular, $F\in\mZ$.
\end{proof}

We summarize the previous results as the following corollary, \textcolor{black}{which gives in turn a new $q$-cryptomorphism}.

\begin{corollary}\label{thm:axiomscheme}\label{thm:cryptomorphism}
Let $(E,r)$ be a $q$-matroid and $(\mZ,\leq,\vee,\wedge)$ be a lattice of subspaces of $E$.
\begin{enumerate}
    \item[(1)] If $\mZ$ is the lattice of cyclic flats of $(E,r)$ then $(\mZ,r)$ satisfies the cyclic flat axioms (Z1)-(Z3). In particular, $(E,r_{\mZ})$ is a $q$-matroid satisfying $r_{\mZ}(A)=r(A)$ for all $A \in \mZ$.
    \item[(2)] If $(\mZ,f)$ satisfies the cyclic flat axioms (Z1)-(Z3) then $\mZ$ is the collection of cyclic flats of the $q$-matroid $(E,f_\mZ)$. In particular, $\mZ=\mZ_{f_\mZ}$.
    \item[(3)] If $\mZ$ is the lattice of cyclic flats of $E$, then $(E,r)=(E,r_\mZ)$.
\end{enumerate}
\end{corollary}

\begin{proof}
   The first statement of Part (1) follows from Theorem \ref{def:axiomscycflats}. The next statement is a consequence of Proposition \ref{prop:restrictionrho} and Theorem \ref{prop:rzsubmod}. Part (2) is the statement of Theorem \ref{thm:proofthm}. Part (3) can be deduced by Parts (1) and (2) combined with  Proposition \ref{prop:reconstruction}.
\end{proof}


\section{Rank-Metric Codes and {\em q}-Matroids}\label{sec:rank-metric}
In this section we focus on representable $q$-matroids and establish a connection between the supports of the codewords of a rank-metric code and the open spaces of its associated $q$-matroid.
In the classical theory, matroids were introduced by Whitney in \cite{whitney1935abstract} in order to generalize the notion of independence in linear algebra. Similarly, we show how the concept of $q$-matroidal independence generalizes a notion of independence for an $\F_q$-subspace over $\F_{q^m}$. In our approach, we will refer to the concept of an $[n,k]_{q^m/q}$-system, introduced in \cite{alfarano2021linear,randrianarisoa2020geometric}, which interprets an $\F_{q^m}$-linear rank-metric code as a geometrical object.

We start by briefly recalling some basic notions on rank-metric codes and explaining a link to $q$-matroids; see \cite{jurrius2018defining}.

Consider the vector space $\F_{q^m}^{n}$, endowed with the following \textbf{rank distance}, defined as $\mathrm{d}_{\mathrm{rk}}(u,v):=\rk(u-v)$ for every $u,v\in\F_{q^m}^n$, where given $v = (v_1,\ldots,v_n)\in\F_{q^m}^n$, we have
$$ \rk(v):=\dim_{\F_q}\langle v_1, \ldots, v_n\rangle _{\F_q}.$$

\begin{definition}
 An $\F_{q^m}$-linear subspace $\mC$ of the metric space $\F_{q^m}^n$ is called a \textbf{rank-metric code}. We say that $\mC$ is an $[n,k]_{q^m/q}$ code, if $k$ is the dimension of $\mC$ over $\F_{q^m}$.
 A matrix $G\in\F_{q^m}^{k\times n}$ of rank $k$ whose rows generate $\mC$ is called a \textbf{generator matrix}. The \textbf{dual code} $\mC^\perp$ of $\mC$ is the $[n,n-k]_{q^m/q}$ rank-metric code comprising vectors orthogonal to all the vectors in $\mC$ with respect to the standard dot product defined by $x \cdot y = \sum_{j=1}^n x_j y_j$ for all $x,y \in \F_{q^m}^n$. Finally, we say that an $[n,k]_{q^m/q}$ rank-metric code $\mC$ is \textbf{non-degenerate} if the columns of any generator matrix for $\mC$ are linearly independent over $\F_q$.
\end{definition}

For a vector
$v \in \F_{q^m}^n$ and an ordered basis $\Gamma=\{\gamma_1,\ldots,\gamma_m\}$ of the field extension
$\F_{q^m}/\F_q$, let $\Gamma(v) \in \F_q^{n \times m}$ be the matrix defined by
$$v_i= \sum_{j=1}^m \Gamma(v)_{ij} \gamma_j.$$
The \textbf{support} of $v$ is the column space of $\Gamma(v)$ for any basis $\Gamma$; we denote it by $\sigma(v)$.

Let $\mC$ be an $[n,k]_{q^m/q}$ rank-metric code and let $G$ be a generator matrix of $\mC$. For every $U\leq \F_q^n$, define the space
$$ \mC_U :=\{v\in\mC \mid \sigma(v)\leq U^\perp\}.$$
Moreover, consider the following function:
\begin{align}\label{eq:rankfunct}
    r : \mL(\F_q^n)\to \Z, \ U \mapsto k - \dim_{\F_{q^m}}(\mC_U).
\end{align}
Note that for any $U\leq \F_q^n$, $r(U) = \rk(G A^U)$, where $A^U$ is a matrix whose columns form a basis of $U$. In fact, $r$ is the same rank function as defined in \eqref{eq:rank1} and it is independent of the choice of generator matrix $G$. $(\F_q^n,r)$ is called the \textbf{$q$-matroid associated to $\mC$} and is denoted by~$M_\mC$.

\begin{definition}
 A $q$-matroid $M=(\F_q^n,r)$ of rank $k$ is \textbf{representable} if there exists some positive integer $m$ and an $[n,k]_{q^m/q}$ rank-metric code $\mC$ such that $M=M_\mC$.
\end{definition}

Non-degenerate rank-metric codes have an equivalent geometric description as $q$-systems; see \cite{alfarano2021linear, randrianarisoa2020geometric}.

\begin{definition}
 An \textbf{$[n,k]_{q^m/q}$  system} is an $n$-dimensional $\Fq$-space $\mU\leq \Fm^k$ with the property that $\langle \mU\rangle_{\Fm}=\Fm^k$.
When the parameters are not relevant, we simply call such an object a \textbf{$q$-system}. Moreover, two $[n,k]_{q^m/q}$  systems $\mU$ and $\mU^\prime$ are equivalent if there exists an $\F_{q^m}$-isomorphism $\phi\in\GL(k,q^m)$ such that $\phi(\mU)=\mU^\prime$.
\end{definition}

There is a correspondence between equivalence classes of non-degenerate $[n,k]_{q^m/q}$ rank-metric codes and equivalence classes of $[n,k]_{q^m/q}$  systems. We briefly explain this connection; for more details we refer the interested reader to \cite{alfarano2021linear, randrianarisoa2020geometric}.

Let $\mC$ be an $[n,k]_{q^m/q}$ non-degenerate rank-metric code with generator matrix $G$. Then the $\F_q$-span $\mU$ of the columns of $G$ is an $[n,k]_{q^m/q}$ system, it is isomorphic to $\F_q^n$ and we call it the \textbf{$q$-system associated to $\mC$}.
Conversely, if $\mU\leq \Fm^k$ is an $[n,k]_{q^m/q}$-system and $G\in\F_{q^m}^{k\times n}$ is a matrix whose columns form a basis of $\mU$, then clearly, the rows of $G$ generate an $[n,k]_{q^m/q}$ rank-metric code.

We recall the following result which provides a natural description of the supports of codewords of $\mC$ in the $q$-system $\mU$ associated to $\mC$.
\begin{theorem}\cite[Theorem 3.1]{neri2021geometry}
Let $\mC$ be an $[n,k]_{q^m/q}$ non-degenerate rank-metric code and let $\mU$ be the $\F_q$-span of the columns of a generator matrix $G$. Consider the isomorphism
\begin{align}\label{eq:psiG}
    \psi_G: \F_q^n \to \mU, \quad  v \mapsto vG^\top.
\end{align}
For every $u\in\F_{q^m}^k$ we have that
$$\psi_G^{-1}(\mU \cap \langle u \rangle ^\perp)=\sigma(uG)^\perp.$$
\end{theorem}

In the following, we use the terminology ``linear basis" of a subspace $V\leq \F_q^n$ to refer to a basis of $V$ as a vector space. This is to distinguish to the notion of basis in the $q$-matroid~ sense.

\begin{definition}\label{def:spaceIndep}
  Let $\mC$ be an $[n,k]_{q^m/q}$ rank-metric code with generator matrix $G$ and $\mU$ be the $\F_q$-span of the columns of $G$. An $\F_q$-subspace $V\leq \mU$ is called \textbf{$\F_{q^m}$-independent} if
  $$ \dim_{\F_{q^m}}(V\otimes_{\F_q} \F_{q^m}) = \dim_{\F_q}(V),$$
  i.e. the vectors in one (and hence in any) linear basis of $V$ are linearly independent over $\F_{q^m}$.
\end{definition}

Thanks to Definition \ref{def:spaceIndep}, we immediately obtain the $q$-analogue of a well-known result in classical matroid theory; see for instance \cite[Theorem 1.1.1]{oxley}.

\begin{theorem}\label{thm:independentrank}
Let $G\in\F_{q^m}^{k\times n}$ be a full rank matrix whose columns are linearly independent over $\F_q$ and let $\mU$ be the $\F_q$-span of the columns of $G$. Let $\psi_G: \F_q^n \to \mU,\   v \mapsto vG^\top$.
Let $M[G]=(\F_q^n,r)$ and let $\I_r$ be the collection
of independent spaces of $M[G]$. Let $\I:=\{I\leq \mU \mid I \textnormal{ is } \F_{q^m}\textnormal{-independent}\}$. We have that
\begin{enumerate}
    \item  $(\mU,\I)$ is a $q$-matroid,
    \item $\{\psi_G(I)\leq \mU \mid I \in \I_r \} = \I.$
\end{enumerate}
\end{theorem}
\begin{proof}
Since the columns of $G$ are linearly independent over $\F_q$, we have that $\psi_G(\F_q^n) = \mU$ and so $\psi_G$ is an isomorphism.
Let $r':\mU\to\Z, \  U\to \rk\left(GA^{\psi_G^{-1}(U)}\right)$, so $r' = r \circ\psi_G^{-1} $. Clearly, $(\mU,r')$ is a $q$-matroid and indeed is equivalent to $M[G]$. The fact that $\I$ is the collection of independent spaces of $(\mU,r')$ follows from the observation that for every $U\leq \mU$, we have that
\begin{equation*}
    \rk\left(GA^{\psi_G^{-1}(U)}\right) = \dim_{\F_{q^m}}(U \otimes_{\F_q} \F_{q^m}).             \qedhere
\end{equation*}
\end{proof}

Let $\mC$ be the $[n,k]_{q^m/q}$ non-degenerate rank-metric code \textcolor{black}{generated by a matrix $G\in\F_q^{k\times n}$}, and let $M_\mC=(\F_q^n,r)$ be its associated $q$-matroid. Recall that $M_\mC^\ast=M_{\mC^\perp}$, where $\mC^\perp$ is the dual code of $\mC$; see \cite{jurrius2018defining}.
The rest of this section is devoted to establishing the connection between the supports of the codewords of $\mC$ and $M_\mC^\ast$.

\begin{definition}
 A nonzero codeword $c\in\mC$ is called \textbf{minimal} if for every $c^\prime \in\mC$ we have that
 $$ \sigma(c^\prime)\leq \sigma(c) \Leftrightarrow c=\lambda c^\prime, \ \textnormal{ for some } \lambda\in\F_{q^m}^\ast.$$
\end{definition}

\begin{lemma}\label{lem:dependent}
For every codeword $v\in\mC^\perp$, the support $\sigma(v)\leq \F_q^n$ is a dependent space of $M_\mC$.
\end{lemma}
\begin{proof}
Let $v=(v_1,\ldots,v_n)\in\mC^\perp$ and let $N \in \F_q^{t\times n}$ have rowspace $\sigma(v)$.
Let $A\in\GL(n,q)$ be a suitable invertible matrix such that $\sigma(vA) = \langle e_1,\ldots,e_t\rangle_{\F_q}$, in which case $vA$ belongs to the dual of the code generated by $G(A^\top)^{-1}$. From basic linear algebra, we have that $\sigma(vA) = A^\top \sigma(v)$.
 Then
\begin{align*}
     r(\sigma(v)) = \rk(G N^\top) 
     = \rk(G (A^T)^{-1} (NA)^T)
     = \rk  \left(G(A^\top)^{-1}\left(\begin{array}{c}
         \mathrm{Id}_t \\
           0
     \end{array} \right)\right) <t.
\end{align*}
Hence, $\sigma(v)$ is a dependent space in the $q$-matroid $M_{\mC}$.
\end{proof}

A different proof for  Lemma \ref{lem:dependent} has been provided in \cite{byrne2021weighted}, which makes use of the characteristic polynomial of the $q$-matroid $M_{\mC}$.

\begin{remark}
The converse of Lemma \ref{lem:dependent} is in general not true.
For example, let $v\in \mC^\perp$ such that $\rk_{\F_q}(v)$ is maximal over all members of $\mC^\perp$. Let $U$ be any subspace of $\Fq^n$ that properly contains $\sigma(v)$. Then $U$ is a dependent space of $M_{\mC}$ since it contains $\sigma(v)$, which is dependent by Lemma \ref{lem:dependent}. But by our choice of $v$, there is no word $u \in \mC^\perp$ with $\sigma(u)=U.$
\end{remark}

However, we can say that every dependent space $D$ of $M_\mC$ contains the support of a codeword of $\mC^\perp$.
In particular, we have the following result, which immediately follows from Lemma \ref{lem:dependent} and the definitions of circuit and minimal codeword.

\begin{lemma}\label{lem:mincodcircuits}
A subspace $C\leq \F_q^n$ is a circuit of $M_{\mC}$ if and only if $C$ is the support of a minimal codeword in $\mC^\perp$.
\end{lemma}

Recall also the following useful result.
\begin{lemma}\cite[Lemma 3.3]{jurrius2017defining}\label{lem:Fromrelinde}
For every $U\leq\F_q^n$, $c \in\mC_U$ if and only if $c \cdot x= 0$ for all~$x\in U$.
\end{lemma}

\begin{theorem}\label{thm:cycsupp}
Let $c\in\mC^\perp$ and let $V=\sigma(c)$. Then $V$ is a cyclic space in $M_{\mC}$.
\end{theorem}
\begin{proof}
Let $c=(c_1,\ldots, c_n)\in\mC^\perp$ and assume that $\dim_{\F_q}\sigma(c)=t$. Up to multiplying the code $\mC$ by an invertible matrix, as in the proof of Lemma \ref{lem:dependent}, we may assume that $V=\langle e_1,\ldots,e_t\rangle_{\F_q}$.
By Lemma \ref{lem:dependent}, $V$ is a dependent space of $M$.
To show that $V$ is cyclic, we need to prove that
$$ \dim_{\F_{q^m}}(\mC_V) = \dim_{\F_{q^m}}(\mC_U),$$
for every $U\in\Hyp(V)$.
Clearly, for any such hyperplane $U$ we have $\mC_V\subseteq \mC_U$, since $V^\perp\leq U^\perp$.
We will show that we have equality.
\noindent
Let $U=\langle u_1,\ldots, u_{t-1}\rangle_{\F_q}$ and choose an isomorphism $\varphi: U \mapsto \langle e_1,\ldots, e_{t-1}\rangle_{\F_q}$. Notice that $\varphi$ extends to an isomorphism of $\F_q^n$, that fixes $V$.
Since $\mC_U=\mC_V$ if and only if $\mC_{\varphi(U)}=\mC_{\varphi(V)} = \mC_V$, we may assume that
$U=\langle e_1,\ldots, e_{t-1}\rangle_{\F_q}$.
Assume, towards a contradiction, that there exists a codeword $g$ in $\mC$, whose support $\sigma(g)$ is contained in $U^\perp$ but not in $V^\perp$. By Lemma \ref{lem:Fromrelinde}, this implies that for every $u\in U$, we have that $g\cdot u=0$ and there exists a vector $v\in V$, $v\not\in U$ such that $g\cdot v\ne 0$. Write such a vector $v$ as $v=u+ae_t$, where $u\in U$ and $a\in\F_q^\ast$, in which case, we have:
\begin{align}\label{eq:rcomponent}
    0\ne g\cdot v = g\cdot (u+ae_t) = g\cdot u+g\cdot ae_t = ag_t.
\end{align}
Since $c\in\mC^\perp$ and $g\in \mC$, we have that
$$0=g\cdot c = \sum_{i=1}^n g_ic_i = \sum_{i=1}^t g_ic_i = g_tc_t+\sum_{i=1}^{t-1}g_ic_i = g_tc_t,$$
where the last equality follows from the assumption that $g\in\mC_U$. Since  $\mC$ is non-degenerate, from \eqref{eq:rcomponent} we have that $ag_t\ne 0$ and so we arrive at a contradiction.
This shows that the $\mC_V = \mC_U$ for every subspace $U\in\Hyp(V)$ and, in other words, the support of every codeword in the dual code of $\mC$ is a cyclic space in the matroid $M$.

\end{proof}

\begin{corollary}\label{cor:mincodewords}
Let $c\in\mC$ be any codeword and let $\{c_1,\ldots,c_t\}\subseteq \mC$ be the set of all minimal codewords in $\mC$ whose supports are contained in $\sigma(c)$. Then
$$\sigma(c) = \sum_{i=1}^t\sigma(c_i).$$
\end{corollary}
\begin{proof}
Let $c\in\mC$ be a codeword. Then, by Theorem \ref{thm:cycsupp}, we have that $\sigma(c)$ is a cyclic space in $M_{\mC^\perp}$. \textcolor{black}{Hence, by Proposition \ref{thm:cycequivalence}, $\sigma(c)$ is the sum of the circuits of $M_{\mC^\perp}$ contained in it.} Finally, by Lemma \ref{lem:mincodcircuits}, we have that each circuit in $M_{\mC^\perp}$ is the support of a minimal codeword in~$\mC$.
\end{proof}

Note that the converse of Theorem \ref{thm:cycsupp} is not true in general, as Example \ref{ex:rankfinal} shows.

\begin{example}\label{ex:rankfinal}
Let $\C$ be the $[5,3]_{2^3/2}$ rank-metric code with generator matrix
$$ G=\begin{pmatrix}
1& \alpha & 1&0& \alpha^2 \\
0 & 1 & \alpha^5 & \alpha^2 & \alpha \\
0 & 0  & 1 & \alpha^4 & \alpha
\end{pmatrix},$$
where $\alpha^3=\alpha +1$. The dual code $\mC^\perp$ is a $[5,2]_{2^3/2}$ code. Let ${M}_\mC=(\F_2^5,r)$ be the $q$-matroid associated to $\mC$. Note that if $c,c'\in\mC^\perp$ and $c=\lambda c'$ for some $\lambda\in\F_{2^3}^\ast$, then $\sigma(c)=\sigma(c^\prime)$. Hence, there are at most $9$ different supports for the nonzero codewords of $\mC^\perp$. Here, we can check that there are indeed $9$ distinct supports for the codewords of $\mC^\perp$ and those are all cyclic spaces in ${M}_\mC$, by Theorem \ref{thm:cycsupp}.
Moreover, there are $11$ cyclic spaces in ${M}_\mC$, which are listed below according to their dimension.

\begin{align*}
\textnormal{Dimension } 0: \quad & A_0:=0. \\
\textnormal{Dimension } 2: \quad & A_1:=\langle e_1+e_4+e_5, e_2+e_4\rangle, \ A_2:=\langle e_1+e_3,e_4\rangle,  \\
& A_3:=\langle e_3+e_5, e_1+e_2+e_4+e_5\rangle. \\
 \textnormal{Dimension } 3:  \quad  & A_4:=\langle e_3+e_5, e_2, e_4 \rangle, \ A_5:=\langle e_2, e_1+e_5, e_3+e_4+e_5 \rangle, \\
 & A_6:=\langle e_1+e_5, e_2+e_4, e_3+e_4+e_5\rangle, \ A_7:=\langle e_1+e_5, e_2, e_3+e_5\rangle, \\
 &A_8:=\langle e_1+e_4+ e_5, e_2, e_3+e_4+e_5\rangle, \ A_9:=\langle e_1+e_5, e_2+e_3+e_5, e_4 \rangle. \\
  \textnormal{Dimension } 4:  \quad  & A_{10}:=\langle e_1+e_5, e_2, e_3+e_5, e_4\rangle.
\end{align*}

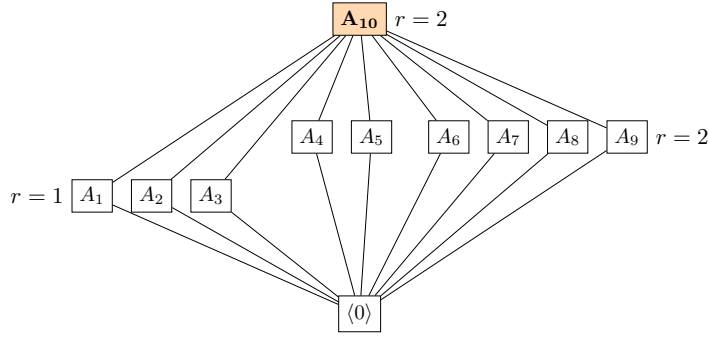
\begin{figure}[ht!]
\centering
\resizebox{0.6\textwidth}{!}{%
\begin{tikzpicture}
\node[shape=rectangle, draw=black] (00) at (0,0) {\small$ \langle 0 \rangle $};
 \node[shape=rectangle,draw=black, label=left:{$r = 1$}] (12) at (-4.5,2) {\small\(A_1\)};
 \node[shape=rectangle,draw=black] (13) at (-3.5,2) {\small\(A_2\)};
 \node[shape=rectangle,draw=black] (14) at (-2.5,2) {\small\(A_3\)};
 \node[shape=rectangle,draw=black] (15) at (-0.80,3) {\small\(A_4\)};
 \node[shape=rectangle,draw=black] (16) at (0.2,3) {\small\(A_5\)};
\node[shape=rectangle,draw=black] (17) at (1.5,3) {\small\( {A_6}\)};
 \node[shape=rectangle,draw=black] (18) at (2.5,3) {\small\(A_7\)};
 \node[shape=rectangle,draw=black] (19) at (3.5,3) {\small\(A_8\)};
 \node[shape=rectangle,draw=black, label=right:{$r = 2$}] (110) at (4.5,3) {\small\(A_9\)};
 \node[shape=rectangle,draw=black, fill=orange!30, label=right:{$r = 2$}] (20) at (0,5) {\small $\mathbf{A_{10}}$};

 \path [- ] (00) edge (12);
 \path [- ] (00) edge (13);
 \path [- ] (00) edge (14);
 \path [- ] (00) edge (15);
 \path [- ] (00) edge (16);
 \path [- ] (00) edge (17);
 \path [- ] (00) edge (18);
 \path [- ] (00) edge (19);
 \path [- ] (00) edge (110);
 \path [- ] (12) edge (20);
 \path [- ] (13) edge (20);
 \path [- ] (14) edge (20);
 \path [- ] (15) edge (20);
 \path [- ] (16) edge (20);
 \path [- ] (17) edge (20);
 \path [- ] (18) edge (20);
 \path [- ] (19) edge (20);
 \path [- ] (110) edge (20);
\end{tikzpicture}
}
\caption{The lattice of cyclic spaces of ${M}_\mC$ from Example \ref{ex:rankfinal}.}
\label{fig:latticecycles}
\end{figure}

In Figure \ref{fig:latticecycles}, the space $A_{10}$ has been highlighted because it is the only cyclic space that is not the support of any codeword in $\mC^\perp$. Finally, note that the all the cyclic spaces except from $A_0$ and $A_{10}$ are circuits. Indeed, one can easily observe that all the codewords in $\mC^\perp$ are minimal and their support is exactly one of the spaces $A_1,\ldots,A_{9}$.
With the aid of \textsc{Magma}, we checked that the $q$-matroid ${M}_\mC$ has $88$ flats and that among these only $5$ are also cyclic, namely $A_0, A_1, A_2, A_3$ and $A_{10}$.

Now, consider the matroid ${M}_{\mC^\perp}$, associated to the code $\mC^\perp$ and denote its rank function by~$r^\ast$. First of all note that ${M}_{\mC^\perp}$ contains the loop $L:=\langle e_1+e_3+e_5\rangle$, which is itself cyclic. In ${M}_{\mC^\perp}$ there are in total $88$ cyclic spaces and among them $65$ are also supports of codewords in $\mC$.
Finally, there are $11$ flats in ${M}_{\mC^\perp}$, which are listed below.

\begin{align*}
\textnormal{Dimension } 1: \quad & F_0:=\langle e_1+e_3+e_5\rangle. \\
\textnormal{Dimension } 2: \quad & F_1:=\langle e_1+e_3+e_5, e_4\rangle, \ F_2:=\langle e_1+e_3+e_5,e_4+e_5\rangle,  \\
& F_3:=\langle e_1+e_3+e_5, e_2+e_3\rangle, \ F_4:=\langle e_3+e_5, e_1 \rangle,\\
&F_5:=\langle e_1+e_3+e_5, e_2+e_3+e_4 \rangle, \ F_6:=\langle e_1+e_4+e_5, e_3+e_4, e_3+e_4\rangle. \\
 \textnormal{Dimension } 3:  \quad  &   F_7:=\langle e_1+e_5, e_2+e_4+e_5, e_3\rangle, \ F_8:=\langle e_1+e_3, e_2, e_5 \rangle , \\ &F_9:=\langle e_1+e_4, e_2+e_4, e_3+e_4+e_5\rangle. \\
  \textnormal{Dimension } 5:  \quad  & F_{10}:=\F_2^5.
\end{align*}
The lattice of flats of ${M}_{\mC^\perp}$ can be found in Figure \ref{fig:latticeflats}.

\begin{figure}[ht!]
\centering
\resizebox{0.6\textwidth}{!}{%
\begin{tikzpicture}
\node[shape=rectangle, draw=black] (00) at (0,0) {\small$ F_0 $};
 \node[shape=rectangle,draw=black, label=left:{$r = 1$}] (12) at (-4.5,2) {\small\(F_1\)};
 \node[shape=rectangle,draw=black] (13) at (-3.5,2) {\small\(F_2\)};
 \node[shape=rectangle,draw=black] (14) at (-2.5,2) {\small\(F_3\)};
 \node[shape=rectangle,draw=black] (15) at (-1.5,2) {\small\(F_4\)};
 \node[shape=rectangle,draw=black] (16) at (-0.5,2) {\small\(F_5\)};
\node[shape=rectangle,draw=black] (17) at (0.8,2) {\small\( {F_6}\)};
 \node[shape=rectangle,draw=black] (18) at (2.5,3) {\small\(F_7\)};
 \node[shape=rectangle,draw=black] (19) at (3.5,3) {\small\(F_8\)};
 \node[shape=rectangle,draw=black, label=right:{$r = 1$}] (110) at (4.5,3) {\small\(F_9\)};
 \node[shape=rectangle,draw=black, label=right:{$r = 2$}] (20) at (0,5) {\small $\F_2^5$};

 \path [- ] (00) edge (12);
 \path [- ] (00) edge (13);
 \path [- ] (00) edge (14);
 \path [- ] (00) edge (15);
 \path [- ] (00) edge (16);
 \path [- ] (00) edge (17);
 \path [- ] (00) edge (18);
 \path [- ] (00) edge (19);
 \path [- ] (00) edge (110);
 \path [- ] (12) edge (20);
 \path [- ] (13) edge (20);
 \path [- ] (14) edge (20);
 \path [- ] (15) edge (20);
 \path [- ] (16) edge (20);
 \path [- ] (17) edge (20);
 \path [- ] (18) edge (20);
 \path [- ] (19) edge (20);
 \path [- ] (110) edge (20);
\end{tikzpicture}
}
\caption{Lattice of flats of the matroid ${M}_{\mC^\perp}$ from Example \ref{ex:rankfinal}.}
\label{fig:latticeflats}
\end{figure}
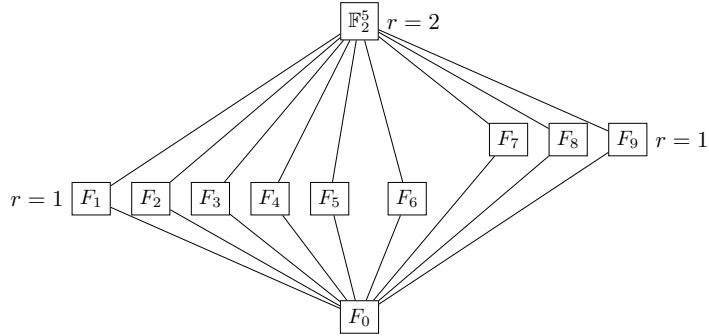
It is not difficult to see that all the flats in Figure \ref{fig:latticeflats} are the orthogonal complements of the cyclic spaces of Figure \ref{fig:latticecycles}, as Proposition \ref{thm:cycequivalence} states. Clearly, the cyclic flats of $M_{\mC^\perp}$ are the orthogonal complement of the cyclic flats in ${M}_\mC$; see Figure \ref{fig:lattcyclicflats}. In particular, $A_{10}^\perp = F_0$, $A_1^\perp=F_7$, $A_2^\perp=F_8$ and $A_3^\perp=F_9$.

Finally, using \textsc{Magma},  we found that there are exactly $33$ minimal codewords in $\mC$. The supports of these codewords are circuits in ${M}_{\mC^\perp}$. Consider $Z:=\langle e_1+e_5, e_2+e_5, e_3+e_5, e_4+e_5\rangle$ and observe that $Z$ is a cyclic space that is not a circuit (since its rank is $2$ and it is not the support of a minimal codeword of $\mC$). $Z$ contains exactly $3$ circuits, namely $\langle e_1+e_3, e_2+e_5 \rangle$, $\langle e_2+e_3+e_4+e_5, e_1+e_5 \rangle$ and $\langle e_1+e_4, e_2+e_4 \rangle$. It is easy to see that $Z$ is actually equal to the sum of the circuits it contains, as stated in Corollary \ref{cor:mincodewords}.

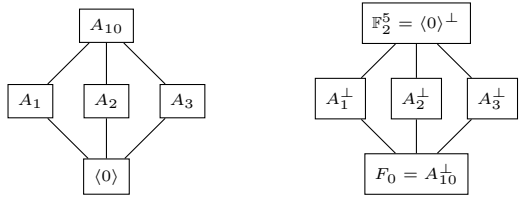
\begin{figure}[ht!]
\centering
\begin{subfigure}[b]{0.25\textwidth}
\centering
\begin{tikzpicture}
\node[shape=rectangle, draw=black] (00) at (0,0) {\tiny$\langle 0 \rangle$};
 \node[shape=rectangle,draw=black] (18) at (-1,1) {\tiny\(A_1\)};
 \node[shape=rectangle,draw=black] (19) at (0,1) {\tiny\(A_2\)};
 \node[shape=rectangle,draw=black] (110) at (1,1) {\tiny\(A_3\)};
 \node[shape=rectangle,draw=black] (20) at (0,2) {\tiny $A_{10}$};
 \path [- ] (00) edge (18);
 \path [- ] (00) edge (19);
 \path [- ] (00) edge (110);
 \path [- ] (18) edge (20);
 \path [- ] (19) edge (20);
 \path [- ] (110) edge (20);

\end{tikzpicture}
\end{subfigure}
\begin{subfigure}[b]{0.25\textwidth}
\centering
\begin{tikzpicture}
\node[shape=rectangle, draw=black] (001) at (0,0) {\tiny$ F_0=A_{10}^\perp$};
 \node[shape=rectangle,draw=black] (181) at (-1,1) {\tiny\(A_1^\perp\)};
 \node[shape=rectangle,draw=black] (191) at (0,1) {\tiny\(A_2^\perp\)};
 \node[shape=rectangle,draw=black] (1101) at (1,1) {\tiny\(A_3^\perp\)};
 \node[shape=rectangle,draw=black] (201) at (0,2) {\tiny $\F_2^5=\langle0\rangle^\perp$};

 \path [- ] (001) edge (181);
 \path [- ] (001) edge (191);
 \path [- ] (001) edge (1101);
 \path [- ] (181) edge (201);
 \path [- ] (191) edge (201);
 \path [- ] (1101) edge (201);
 \end{tikzpicture}
\end{subfigure}
\caption{The lattices of the cyclic flats of the $q$-matroids ${M}_\mC$ and ${M}_{\mC^\perp}$ from Example \ref{ex:rankfinal}.}
\label{fig:lattcyclicflats}
\end{figure}
\end{example}


\section{Final Remarks}\label{subsec:q-polymatroids}
$q$-Polymatroids and their connections to rank-metric codes were introduced in \cite{gorla2019rank} and \cite{shiromoto2019codes}.
In \cite{csirmaz2020cyclic}, it was shown that knowledge of the lattice of cyclic flats, along with the ranks of its elements, is sufficient to determine a polymatroid.
It is a natural question to ask whether or not a $q$-analogue of this result holds. We propose that an answer to this question may require more than a straightforward extension of the results presented here.

\begin{definition}[\cite{shiromoto2019codes}]
A $(q,r)$-polymatroid is a pair $M=(E, \rho)$ for which $r \in \Z$ and $\rho$ is an integer-valued function on the subspaces of $E$, satisfying the following axioms.
\begin{itemize}
	\item[($\rho$1)] $0\leq \rho(A) \leq r\dim A$, for all $A\in\mL(E)$.
	\item[($\rho$2)] $A\leq B \Rightarrow\rho(A)\leq \rho(B)$, for all $A,B\in\mL(E)$.
	\item[($\rho$3)] $\rho(A+B)+\rho(A\cap B)\leq \rho(A)+\rho(B)$, for all $A,B\in\mL(E)$.
\end{itemize}
Note that a $(q,1)$-polymatroid is a $q$-matroid.
\end{definition}

We can define cyclic spaces along with the cyclic operator as follows.

\begin{definition}\label{def:cyclic_qpol}
 For each $A\in\mL(E)$, define
 \begin{align*}
     \Cyc_\rho(A):=\{& x \leq A:   \rho(x)=0 \textnormal{ or } \rho(B+x)-\rho(B) <\rho(x),\textnormal{ for all } B\in\Hyp(A)\},
 \end{align*}
 and $\cyc_\rho(A)$ as the vector space sum of the cyclic spaces contained in $A$.
\end{definition}

In \cite{csirmaz2020cyclic} a rank function on the lattice of cyclic flats is defined via a convolution of the rank function of the polymatroid with a modular function on the underlying Boolean lattice (see Remark \ref{rem:convolution_submodula_modular}), which is an important device in obtaining the cryptomorphism in the polymatroid case.
It is known that a function $\mu: \mL(E) \to \Z$ is modular if and only if $\mu(\langle0\rangle)=0$ and $\mu$ is additive on $\mL(E)$, that is, if $\mu(X+Y)=\mu(X)+\mu(Y)$ for all $X,Y\in \mL(E)$ satisfying $X\cap Y=\{0\}$.
However, as the following result shows, the only additive function on $\mL(E)$ is a constant multiple of the dimension function.

\begin{proposition}\label{prop:additive}
Let $\mu:\mL(E) \to \Z$ be an additive function. Then there exists an integer $\ell \in \Z$ such that $\mu(A)=\ell \dim(A)$ for all $A \in \mL(E) $.
\end{proposition}
\begin{proof}
Let $x,y\in\mL(E)$ and choose $A\in\Hyp(E)$ that contains neither $x$ nor $y$. Clearly, $E=A+x=A+y$ and by the additivity of $\mu$ we have
$$ \mu(A)+\mu(x) = \mu(A+x)= \mu(E)=\mu(A+y) = \mu(A)+\mu(y),$$
which implies that $\mu$ is constant on all the one-dimensional subspaces of $E$. Hence there exists a constant $\ell\in\Z$, such that $\mu(A)=\ell \dim(A)$, for every $A\in\mL(E)$.
\end{proof}

Note that modular functions defined on a boolean lattice are not necessarily constant multiples of the cardinality function.
Generalizing the results of Section \ref{sec:cryptomorphism} to the $q$-polymatroid case would be possible if an analogue of \cite[Lemma 1]{csirmaz2020cyclic} could be obtained. Proposition \ref{prop:additive} implies that a $q$-analogue would be given by showing that for every $A\in\mL(E)$ we have that
$$ \rho(A) - \rho(\cyc_\rho(A)) = r\cdot (\dim(A)-\dim(\cyc_\rho(A)).$$

However, the following example shows that this is not always the case.

\begin{example}\label{ex:3x3}
Let $\mC\leq \F_3^{3\times 3}$ be the space of $3\times 3$ matrices generated by
\begin{align*}
    M_1:=\begin{pmatrix}
    1 & 2 &1 \\
    1& 2 & 2 \\
    1 & 1 & 1
    \end{pmatrix}, \   M_2:=\begin{pmatrix}
    0 & 2 &1 \\
    2& 1 & 1 \\
    0 & 1 & 2
    \end{pmatrix},  \   M_3:=\begin{pmatrix}
    2 & 2 &0 \\
    1& 1 & 1 \\
    0 & 1 & 2
    \end{pmatrix},  \   M_4:=\begin{pmatrix}
    0 & 2 &2 \\
    0& 0 & 0 \\
    2 & 0 & 1
    \end{pmatrix}.
\end{align*}
Consider the function
$$\rho : \mL(\F_3^3) \to \Z, \ A\mapsto 4-\dim_{\F_3}(\mC_A),$$ where  $\mC_A$ is the subspace of $\mC$ made of the matrices whose column span is contained in $A^\perp$.
From \cite{shiromoto2019codes}, $M=(\F_3^3,\rho)$ is a $(q,3)$-polymatroid.
Consider the $1$-dimensional space $ c=\langle (1,1,2) \rangle$ whose rank is $2$. It is easy to see that
$$ 2 = \rho(c) - \rho(\cyc_\rho(c)) \ne 3(\dim(c) - \dim(\cyc_\rho(c)) = 3.$$
\end{example}

 \bigskip

\section*{Acknowledgments}
The work of Gianira N. Alfarano  was supported by the Swiss National Science Foundation through grants no. 188430 and no. 210966.
The authors are thankful to Alessandro Neri for fruitful discussions.

\bigskip

\bibliographystyle{abbrv}
\bibliography{references.bib}

\begin{thebibliography}{10}

\bibitem{alfarano2021linear}
G.~N. Alfarano, M.~Borello, A.~Neri, and A.~Ravagnani.
\newblock Linear cutting blocking sets and minimal codes in the rank metric.
\newblock {\em J. Comb. Theory Ser. A}, 192:105658, 2022.

\bibitem{bonin2008lattice}
J.~E. Bonin and A.~De~Mier.
\newblock The lattice of cyclic flats of a matroid.
\newblock {\em Ann. Comb.}, 12(2):155--170, 2008.

\bibitem{MR1484478}
W.~Bosma, J.~Cannon, and C.~Playoust.
\newblock The {M}agma algebra system. {I}. {T}he user language.
\newblock {\em Journal of Symbolic Computation}, 24(3-4):235--265, 1997.
\newblock Computational algebra and number theory (London, 1993).

\bibitem{brylawski1975affine}
T.~H. Brylawski.
\newblock An affine representation for transversal geometries.
\newblock {\em Stud. Appl. Math.}, 54(2):143--160, 1975.

\bibitem{byrne2021weighted}
E.~Byrne, M.~Ceria, S.~Ionica, and R.~Jurrius.
\newblock Weighted subspace designs from $ q $-polymatroids.
\newblock {\em arXiv preprint arXiv:2104.12463}, 2021.

\bibitem{byrne2022constructions}
E.~Byrne, M.~Ceria, S.~Ionica, R.~Jurrius, and E.~Sa{\c{c}}{\i}kara.
\newblock Constructions of new matroids and designs over $\mathbb{F}_q$.
\newblock {\em Des. Codes, Cryptogr.}, pages 1--23, 2022.

\bibitem{byrne2020cryptomorphisms}
E.~Byrne, M.~Ceria, and R.~Jurrius.
\newblock Constructions of new $q$-cryptomorphisms.
\newblock {\em J. Comb. Theory Ser. B}, 153:149--194, 2022.

\bibitem{ceria2021direct}
M.~Ceria and R.~Jurrius.
\newblock The direct sum of $ q $-matroids.
\newblock {\em arXiv preprint arXiv:2109.13637}, 2021.

\bibitem{crapo1964theory}
H.~Crapo.
\newblock {\em On the theory of combinatorial independence}.
\newblock PhD thesis, Massachusetts Institute of Technology, Department of
  Mathematics, 1964.

\bibitem{csirmaz2020cyclic}
L.~Csirmaz.
\newblock Cyclic flats of a polymatroid.
\newblock {\em Ann. Comb.}, 24(4):637--648, 2020.

\bibitem{eberhardt2014computing}
J.~N. Eberhardt.
\newblock Computing the {T}utte polynomial of a matroid from its lattice of
  cyclic flats.
\newblock {\em Electron. J. Comb.}, 21(3):P3--47, 2014.

\bibitem{freij2021cyclic}
R.~Freij-Hollanti, M.~Grezet, C.~Hollanti, and T.~Westerb{\"a}ck.
\newblock Cyclic flats of binary matroids.
\newblock {\em Adv. Appl. Math.}, 127:102165, 2021.

\bibitem{ghorpade2020polymatroid}
S.~R. Ghorpade and T.~Johnsen.
\newblock A polymatroid approach to generalized weights of rank metric codes.
\newblock {\em Des. Codes, Cryptogr.}, 88(12):2531--2546, 2020.

\bibitem{gorla2019rank}
E.~Gorla, R.~Jurrius, H.~H. L{\'o}pez, and A.~Ravagnani.
\newblock Rank-metric codes and $q$-polymatroids.
\newblock {\em J. Algebr. Comb.}, 52:1--19, 2020.

\bibitem{johnsen}
T.~Johnsen, R.~Pratihar, and H.~Verdure.
\newblock Weight spectra of {G}abidulin rank-metric codes and {B}etti numbers.
\newblock {\em Sao Paulo Journal of Mathematical Sciences,}, 2022.

\bibitem{jurrius2017defining}
R.~Jurrius and G.~Pellikaan.
\newblock On defining generalized rank weights.
\newblock {\em Adv. Math. Commun.}, 11(1):225--235, 2017.

\bibitem{jurrius2018defining}
R.~Jurrius and G.~Pellikaan.
\newblock Defining the $q$-analogue of a matroid.
\newblock {\em Electron. J. Comb.}, 25(3), 2018.

\bibitem{lovasz1983submodular}
L.~Lov{\'a}sz.
\newblock Submodular functions and convexity.
\newblock In {\em Mathematical programming the state of the art}, pages
  235--257. Springer, 1983.

\bibitem{neri2021geometry}
A.~Neri, P.~Santonastaso, and F.~Zullo.
\newblock The geometry of one-weight codes in the sum-rank metric.
\newblock {\em J. Comb. Theory Ser. A}, 194:105703, 2023.

\bibitem{oxley}
J.~Oxley.
\newblock {\em Matroid Theory}.
\newblock Oxford University Press, second edition, 2011.

\bibitem{randrianarisoa2020geometric}
T.~H. Randrianarisoa.
\newblock A geometric approach to rank metric codes and a classification of
  constant weight codes.
\newblock {\em Des. Codes, Cryptogr.}, 88(7):1331--1348, 2020.

\bibitem{shiromoto2019codes}
K.~Shiromoto.
\newblock Codes with the rank metric and matroids.
\newblock {\em Des. Codes, Cryptogr.}, 87(8):1765--1776, 2019.

\bibitem{sims1977extension}
J.~A. Sims.
\newblock An extension of {D}ilworth's theorem.
\newblock {\em J. London Math. Soc.}, 2(3):393--396, 1977.

\bibitem{westerback2016combinatorics}
T.~Westerb{\"a}ck, R.~Freij-Hollanti, T.~Ernvall, and C.~Hollanti.
\newblock On the combinatorics of locally repairable codes via matroid theory.
\newblock {\em IEEE Trans. Inf. Theory}, 62(10):5296--5315, 2016.

\bibitem{westerback2015applications}
T.~Westerb{\"a}ck, R.~Freij-Hollanti, and C.~Hollanti.
\newblock Applications of polymatroid theory to distributed storage systems.
\newblock In {\em 2015 53rd Annu. Allerton Conf. Commun. Control Comput.
  (Allerton)}, pages 231--237. IEEE, 2015.

\bibitem{westerback2018polymatroidal}
T.~Westerb{\"a}ck, M.~Grezet, R.~Freij-Hollanti, and C.~Hollanti.
\newblock On the polymatroidal structure of quasi-uniform codes with
  applications to heterogeneous distributed storage.
\newblock In {\em International Symposium on Mathematical Theory of Networks
  and Systems}, pages 641--647, 2018.

\bibitem{whitney1935abstract}
H.~Whitney.
\newblock On the abstract properties of linear dependence.
\newblock {\em Am. J. Math.}, 57(3):509--533, 1935.

\end{thebibliography}

\end{document}